\documentclass[12pt]{amsart}
\usepackage[margin=1in]{geometry}
\usepackage{amsmath,amsfonts,amsthm,amssymb,bbm}
\usepackage{graphicx,color,dsfont}
\usepackage{enumitem}
\usepackage{fourier}

\usepackage{mathrsfs}

\newtheorem{theorem}{Theorem}[section]
\newtheorem{lemma}{Lemma}[section]
\newtheorem{proposition}{Proposition}[section]

\newtheorem{rmk}{Remark}[section]

\newtheorem{definition}{Definition}[section]

	\newtheorem{corollary}{Corollary}[section]

\newtheorem{claim}{Claim}[section]

 \theoremstyle{definition}
 
 \theoremstyle{remark}

 \numberwithin{equation}{section}

\newcommand{\vertiii}[1]{{\left\vert\kern-0.25ex\left\vert\kern-0.25ex\left\vert #1
    \right\vert\kern-0.25ex\right\vert\kern-0.25ex\right\vert}}

\newcommand{\R}{{\mathbb R}}
\newcommand{\C}{{\mathbb C}}

\newcommand{\f}[2]{\frac{#1}{#2}}

\newcommand{\cl}{{\mathcal L}}

\newcommand{\al}{\alpha}

\newcommand{\de}{\delta}

\newcommand{\De}{\Delta}

\newcommand{\la}{\lambda}

\newcommand{\si}{\sigma}

\newcommand{\om}{\omega}

\newcommand{\rd}{{\mathbb R}^d}

\newcommand{\rone}{\mathbb R}
\newcommand{\rtwo}{\mathbb R^2}

\newcommand{\dpr}[2]{\langle #1,#2 \rangle}

\newcommand{\eps}{\epsilon}

\newcommand{\cs}{\mathcal S}

\newcommand{\cf}{\mathcal F}

\newcommand{\cd}{\mathcal D}

\newcommand{\ch}{\mathcal H}

\newcommand{\p}{\partial}

\newcommand{\beq}{\begin{equation}}
\newcommand{\eeq}{\end{equation}}
\newcommand{\beqna}{\begin{eqnarray*}}
\newcommand{\eeqna}{\end{eqnarray*}}
\newcommand{\beqn}{\begin{equation*}}
\newcommand{\eeqn}{\end{equation*}}
\newcommand{\bp}{\begin{proof}}
\newcommand{\ep}{\end{proof}}
\newcommand{\bprop}{\begin{proposition}}
\newcommand{\eprop}{\end{proposition}}
\newcommand{\bt}{\begin{theorem}}
\newcommand{\et}{\end{theorem}}
\newcommand{\bex}{\begin{Example}}
\newcommand{\eex}{\end{Example}}
\newcommand{\bc}{\begin{corollary}}
\newcommand{\ec}{\end{corollary}}
\newcommand{\bcl}{\begin{claim}}
\newcommand{\ecl}{\end{claim}}
\newcommand{\bl}{\begin{lemma}}
\newcommand{\el}{\end{lemma}}

\newcommand{\cj}{{\mathcal J}}

\begin{document}

\title[Solitary waves for anisotropic nonlinear Schr\"odinger models]{On the solitary waves for anisotropic nonlinear Schr\"odinger models on the plane}

\author[Tianxiang Gou]{\sc Tianxiang Gou}
\address{ School of Mathematics and Statistics, Xi'an Jiaotong University,
	Xi'an, Shaanxi 710049, People's Republic of China
}
\email{tianxiang.gou@xjtu.edu.cn}

\author[Hichem Hajaiej]{\sc Hichem Hajaiej}
\address{Department of Mathematics, California State University at Los Angeles,
		  Los Angeles, CA 90032, USA
}
\email{hhajaie@calstatela.edu}

\author[Atanas G. Stefanov]{\sc Atanas G. Stefanov}
\address{Department of Mathematics, University of Alabama-Birmingham,
	1402 10th Avenue South
	Birmingham AL 35294, USA
}
\email{stefanov@uab.edu}

 \thanks{{\it{Acknowledgements.}} Tianxiang Gou acknowledges support from the National Natural Science Foundation of China (No. 12101483) and the Postdoctoral Science Foundation of China (No. 2021M702620). Atanas G. Stefanov acknowledges partial support from NSF-DMS, grant number \# 2204788.}




\begin{abstract}
	The focussing anisotropic nonlinear Schr\"odinger equation
	\begin{align*} 
	\textnormal{i} u_t-\partial_{xx} u + (-\p_{yy})^s  u=|u|^{p-2}u \quad  \mbox{in}\,\,\, \R \times \R^2
	\end{align*}
is considered for $0<s<1$ and $p>2$. Here the equation is of anisotropy, it means that dispersion of solutions along $x$-axis and $y$-axis is different. We show that while localized time-periodic waves, that are solutions in the form $u=e^{-\textnormal{i} \om t} \phi$, do not exist in the regime $p\geq p_s:=\f{2(1+s)}{1-s}$, they do exist in the complementary regime $2<p<p_s$. 
In fact, we construct them variationally and we establish a number of key properties.  Importantly, we completely characterize their spectral stability properties. Our consideration are easily extendable to the higher dimensional situation. 

We also show uniqueness of these waves under a natural weak non-degeneracy assumption. This assumption is actually removed for $s$ close to $1$, implying uniqueness for the waves in the full range of parameters. 
\end{abstract}

\maketitle

  \section{Introduction}

   In this paper, we study the focusing anisotropic nonlinear Schr\"odinger equation, posed on $\rtwo$, as follows
   \begin{equation}
   	\label{10}
   	\textnormal{i} u_t	-\partial_{xx} u + (-\p_{yy})^s  u=|u|^{p-2}u \quad  \mbox{in}\,\,\, \R \times \R^2.
   \end{equation}
Here $s\in (0,1), p>2$ and the fractional Laplacian
$(-\p_{yy})^s$ is characterized by $\mathcal{F}((-\p_{yy})^{s}u)(\xi)=|\xi_2|^{2s} \mathcal{F}(u)(\xi)$
for $\xi=(\xi_1, \xi_2) \in \R^2$, where $\mathcal{F}$ denotes the Fourier transform. The equation \eqref{10} was first proposed in \cite{Xu}, where the author discussed scattering of solutions to the Cauchy problem for \eqref{10} set on $\R_x \times \mathbb{T}_y$ with $s=\frac 12$ and $p=4$. Later, in \cite{BIK}, the well-posedness of solutions to Cauchy problem for \eqref{10}, the existence and orbital stability of solitary waves to \eqref{10} with $\frac 12<s<1$ and $2<p<6$ were considered.


A question arises, namely about the  existence, uniqueness and stability of special solutions to \eqref{10}, specifically solitons in the form $u=e^{- \textnormal{i} \om t}\phi$ for $\om>0$. These naturally satisfy, in appropriate sense,  the following elliptic equation
\begin{equation}
	\label{20}
	-\partial_{xx} \phi + (-\p_{yy})^s  \phi +\om \phi=|\phi|^{p-2}\phi \quad \mbox{in} \,\, \, \R^2.
\end{equation}
Here we are actually interested in spectral stability of the waves, which is different from the orbital stability dealt with in \cite{BIK}. More concretely, taking
$u=e^{- \textnormal{i} \om t}(\phi+z)$ with $\phi>0$, pluging it in \eqref{10} and expanding up to $O(z)$ order, we obtain the linear eigenvalue problem
$$
\textnormal{i} z_t+\om z - \p_{xx} z+(-\p_{yy})^s z - (p-2) \phi^{p-2} \Re z - \phi^{p-2} z=0.
$$
Splitting into real and imaginary parts and the eigenvalue ansatz $z(t,x,y)=e^{\la t}(z_1(x,y)+\textnormal{i} z_2(x,y))$, we then derive that
\begin{equation}
	\label{85}
	\left\{   \begin{array}{l}
	-\p_{xx} z_1+(-\p_{yy})^s z_1+\om z_1 - (p-1) \phi^{p-2} z_1 =  \la z_2,\\
	-\p_{xx} z_2+(-\p_{yy})^s z_2+\om z_2 -  \phi^{p-2} z_1 =  - \la z_1.
\end{array}
	 \right.
\end{equation}
In matrix form
\begin{equation}
	\label{90}
	 \left(\begin{array}{cc}
	 	0 & -1 \\ 1 & 0
	 \end{array}\right)  \left(\begin{array}{cc}
	 \cl_{+, s} & 0 \\ 0& \cl_{-,s}
 \end{array}\right) \vec{z} = \la \vec{z},
\end{equation}
where
\begin{align*}
	\cl_{+,s} = -\p_{xx}  +(-\p_{yy})^s +\om - (p-1) \phi^{p-2}, \quad
	\cl_{-,s} = -\p_{xx}  +(-\p_{yy})^s +\om -  \phi^{p-2}.
\end{align*}
By direct inspection, we have that $\cl_{s, -}[\phi]=0$, while a differentiation in the $x,y$ variables reveals that $Ker(\cl_{s, +})\subset span[\phi_x, \phi_y]$. It is relatively easy to establish\footnote{We do so below, see Proposition \ref{prop:30}.}, that $Ker(\cl_{s,-})=span[\phi]$, which is basically due to Perron-Frobenius property enjoyed by $\cl_{s,-}$ and the positivity of $\phi$. 
However, it is much harder to prove that indeed $Ker(\cl_{s, +})=span[\phi_x, \phi_y]$. In fact, this is a well-known difficulty in the theory, because of the anisotropic feature of the problem under consideration, which is  often referred to as a non-degeneracy property of the wave. Such a property has been verified for ground states of the classical problem in \cite{K} for $s=1$ and $2<p<2^*$.

 \subsection{Main results}
 We start with the existence results. To this end, we need to introduce the notion of an axially symmetric function. We say that a function $\phi:\rtwo\to \rone$ is {\it axially symmetric}, if for each $y$ the function $x \mapsto \phi(x, y)$ is
symmetric and for each $x$ the function $y \mapsto \phi(x, y)$ is
symmetric. We shall also denote
 $$
 L^2_a(\rtwo):=\left\{\phi\in L^2(\rtwo): \phi \ \textup{is axially symmetric} \right\}.
 $$
 \begin{theorem}(Existence of solitary waves)
 	\label{theo:10}
 	Let $s\in (0, 1), \om>0$ and $2<p<p_s=\frac{2(1+s)}{1-s}$. Then there exists a positive ground state $\phi$ of the profile equation \eqref{20}, that is axially symmetric.  In fact, the solution $\phi$ coincides with its Steiner rearrangement and specifically for each $x$ and $y$ the functions $\phi(x, \cdot) $ and $\phi(\cdot, y)$ are bell-shaped respectively, i.e. positive, symmetric and decaying in $(0, \infty)$.
 Moreover, such a solution has\footnote{See the definition below in Section \ref{sec:2}.}  $\phi\in H^{2, 2s}(\rtwo)$ and obeys the point-wise estimates
 \begin{equation}
 	\label{h:10}
 	\begin{cases}
 		 	|\p_x^\al \phi(x,y)|\leq C_{\al}  e^{-\sqrt{\om} |x|}  (1+|y|)^{-1-2s},  \,\,\,\al
 		 	\in {\mathbb N}, \\
 		 		|\p_x^\al \p_y \phi(x,y)|\leq C_{\al}  e^{-\sqrt{\om} |x|} (1+|y|)^{-2-2s},  \,\,\,\al \in {\mathbb N},
 	\end{cases}
 \end{equation}
 where $H^{2,2s}(\R^2)$ is defined by the completion of $C^{\infty}_0(\R^2)$ under the norm \eqref{space1}.
 \end{theorem}
 
By a ground state to \eqref{20}, we mean that a solution to \eqref{20} has the least energy among all solutions.

\begin{rmk} \label{re1}
Here we shall make use of the underlying Weinstein functional to establish the existence of the waves. The approach is distinctive from the one adapted in \cite{BIK}, where the waves are obtained as critical points of the energy functions in the energy space. The existence part of the proof of Theorem \ref{theo:10} is variational in nature and proceeds in a number of steps.  First, we construct minimizers of a Weinstein type of variational problem, see Proposition \ref{prop:10}. These are, up to a multiple, solutions to \eqref{20}, see Proposition \ref{prop:30}, along with the various additional spectral properties, which are relevant in the subsequent discussion about the stability of these waves. The decay estimates \eqref{h:10} are established
by the idea that since $\phi$ solves \eqref{20}, then the decay rate for $\phi$ matches the corresponding decay rate of the kernel of the Green's function $\left(\p_{xx}+(-\p_{yy})^s+\om\right)^{-1}$.
\end{rmk}

The next result concerns the necessity of  the conditions $\om>0$ and  $2<p<p_s$.  More precisely, we have the following proposition.
\begin{proposition}
	\label{prop:12}
	Let $s\in (0,1)$, $p>2$ and $\phi$ be a positive solution of \eqref{20}, which has the property $\phi\in L^1_x L^{p-1}_y(\R^2)$ and coincides with its Steiner rearrangement, then $\om>0$. In addition, if $\phi\in H^{1,s}(\R^2)\cap L^p(\rtwo)$, then $2<p<p_s$, where 
	$$
	L^1_x L^{p-1}_y(\R^2):=\left\{\phi : \R^2 \to \R \,\, \mbox{is measurable} : \left\|\left\|u\right\|_{L^{p-1}_y} \right\|_{L^1_x}<\infty\right\}
	$$
	and $H^{1,s}(\R^2)$ is the energy space related to \eqref{10} defined by the completion of $C^{\infty}_0(\R^2)$ under the norm \eqref{space2} or \eqref{space}.
\end{proposition}

\medskip
The following result concerns the stability of the solitary  waves constructed in Theorem \ref{theo:10} in the following sense.

\begin{definition}
We say that the wave $e^{- \textnormal{i} \om t}\phi$ is spectrally stable, if the eigenvalue problem \eqref{90} has no eigenvalue $\lambda \in \C$ with $\Re \lambda>0$. Otherwise, we say that the wave is spectrally unstable.
\end{definition}

\begin{theorem}(Complete description of the stability of soliatry waves)
 	\label{theo:20}
 	Let $0<s<1$, $\om>0$ and $2<p<p_s$. Then the waves described in Theorem \ref{theo:10} are spectrally stable if and only if $2<p\leq p_m:=\f{6s+2}{s+1}$ and spectrally unstable in the complementary range $p_m<p<p_s$.
 	
 \end{theorem}

We now state our uniqueness results. Unfortunately, they come with some limitations.  More precisely, we have the following theorem, whose proof is completed by fully adapting the ideas developed in \cite{FL, FLS}, recently generalized in \cite{HS}.

\begin{theorem}(Conditional Uniqueness)
	\label{uniqueness}
	Let $0<s<1$, $\om>0$ and $2<p<p_s$. Then ground state to \eqref{20} is unique up to translations provided that the linearized operator
	$$
	\mathcal{L}_{+, s}:=-\partial_{xx}  + (-\partial_{yy})^{s}   + \om-(p-1)\phi^{p-2}
	$$
	has trivial kernel on $L^2_{a}(\R^2)$, where $\phi \in H^{1,s}(\R^2)$ is a ground state to \eqref{20}.
	
\end{theorem}
Note that the condition imposed herein, namely $Ker(\cl_{+,s})\cap L^2_a(\R^2)=\{0\}$ is a weaker one than the non-degeneracy property of $\phi$ discussed above. On the other hand, as the weak non-degeneracy holds for $s=1$, it is natural to ask whether uniqueness  holds at least in some neighborhood of $s=1$. It turns out that this is the case, and this is the subject of our next result. 
\begin{theorem} (Uniqueness for $s$ close to $1$)
	\label{nd}
	Let $0<s<1$, $\om>0$ and $2<p<p_s$. Then there exists $0<s_0<1$ such that, for any $s_0<s<1$, the ground state to \eqref{20} is unique up to translations. In addition, there holds that $Ker(\cl_{+,s})=span[\phi_x, \phi_y]$.
\end{theorem}


 \begin{rmk}
 Theorems \ref{uniqueness} and \ref{nd} partially answer open questions raised in \cite{BIK} with respect to uniqueness and non-degeneracy of ground states under certain assumptions. 
 \end{rmk}

We plan our paper as follows. In Section \ref{sec:2}, we introduce some preliminary notions as well as some relevant Gagliardo-Nirenberg type inequalties and Pohozaev identities. As a result, we establish the necessary conditions on the existence of waves, as listed in Proposition \ref{prop:12}. In Section \ref{sec:4}, we start by presenting a variational construction of the waves in the region $2<p<p_s$ via the Weinstein functional approach. Then Theorem \ref{theo:10} is established as a direct consequence of Proposition \ref{prop:30} and Remark \ref{re1}. Subsequently, we outline the Hamiltonian index counting theory, which helps us establish the stability of the waves. Specifically, we show that the stable waves are exactly those corresponding to the range $2<p\leq p_m$. This completes the proof of Theorem \ref{theo:20}.  In Section \ref{unique}, we show the uniqueness of the ground state subject to a weak non-degeneracy condition and present proofs of Theorems \ref{uniqueness} and \ref{nd}.

\section{Preliminaries}
\label{sec:2}
\subsection{Fourier transform, fractional differential operators and functions spaces}
First, we introduce the standard Lebesgue spaces norms
$$
\|f\|_q=\left(\int_{\rd} |u|^q \, dx\right)^{\f{1}{q}}.
$$
In order to further fix notations, we define the Fourier transform and its inverse repespectively as follows
$$
\hat{u}(\xi)=\cf[u](\xi)=\int_{\rd} u(x) e^{-2\pi \textnormal{i} x \cdot \xi }dx; \ \ u(x)=\cf^{-1}[u]=\int_{\rd}  \hat{u}(\xi) e^{2\pi \textnormal{i}  x\cdot \xi } d\xi.
$$
This is of course valid for all  Schwartz functions $f$, with appropriate extensions to $f\in L^1(\rd)$. Regarding the derivative operators, it is clear that $\widehat{\p_{x_j} f}(\xi)=2\pi i \xi_j \hat{f}(\xi), j=1, 2, \ldots, d$, whence $\widehat{\De f}(\xi)=-4\pi^2 |\xi|^2 \hat{f}(\xi)$ for $\xi \in \R^d$. It can be shown that $-\De$ is a non-negative self-adjoint operator, which has well-defined fractional powers. This has the nice realization,
$\widehat{(-\De)^s f}(\xi)=(4\pi^2 |\xi|^2)^s \hat{f}(\xi)$, which is well-defined forat least  all $s>0$. In fact, one can introduce the $L^p$ Sobolev spaces, $W^{s,q}(\R^d), s>0, 1<q<\infty$, via the norms\footnote{These expressions are well-defined for Schwartz functions $f$, but can be extended to a meaningfull expression (and this is the definition of $W^{s,q}(\R^d)$.) to a linear subspace of elements of $L^q(\R^d)$.}
$$
\|f\|_{W^{s,q}(\rd)}:=\| (-\De)^s f\|_q+ \|f\|_q.
$$
Moreover, one defines, and we shall need to do so herein, the operator $D_{x_j}:=\sqrt{-\p_{x_j x_j}}$, or alternatively,
$\widehat{D_{x_j} f}(\xi)=2\pi |\xi_j|  \hat{f}(\xi)$. As we shall work exclusively with the 2D case, we stick with the Fourier correspondence variables $x\to \xi, y\to \eta$. For future reference, we introduce a special notation for the kinetic portion of the energy space $H^{1,s}(\rtwo)$. Namely,
\begin{align} \label{space2}
\|u\|_{H^{1,s}(\rtwo)}:=  \left(\int_{\rtwo} |\partial_x u(x,y)|^2 +  |D_y^s  u(x,y)|^2+ |u(x,y)|^2 \,dx dy\right)^{\f{1}{2}}.
\end{align}
Alternatively, an equivalent norm in terms of the Fourier transform is given by
\begin{align} \label{space}
\|u\|_{H^{1,s}(\rtwo)}:=  \left(\int_{\rtwo}  (|\xi|^2+|\eta|^{2s}+1) |\hat{u}(\xi, \eta)|^2 \,d\xi d\eta \right)^{\f{1}{2}}.
\end{align}
One can also easily define another relevant space
\begin{align} \label{space1}
\|u\|_{H^{2, 2 s}(\rtwo)}:=  \left(\int_{\rtwo}  (|\xi|^2+|\eta|^{2s}+1)^2 |\hat{u}(\xi, \eta)|^2 \,d\xi d\eta \right)^{\f{1}{2}}.
\end{align}

Hereinafter we write $D_y^s:=(-\p_{yy})^{\frac s 2}$ for simplicity. A very important ingredient in this discussion is the following  recent Gagliardo-Nirenberg type inequality established by Esfahani in \cite{E}.

When there is a function $\varphi$ defined in $\R_+\times\R_+$ such that a function $u(x,y)=\varphi(|x|,|y|)$, we say that $u$ is axially symmetric. We denote by $H_a^{1,s}(\R^2)$ the axially symmetric functions in $H^{1,s}(\R^2)$.
\begin{lemma}  \cite[Theorem 1.1]{E}
	\label{le:es}
	Let $0<s<1$ and $2 \leq q \leq p_s$. Then there exists a constant $C_{q,s}>0$ such that
	\begin{align} \label{gn}
		\int_{\R^2}|u|^q \,dx dy \leq C_{q, s} \left(\int_{\R^2}|u|^2 \, dxdy \right)^{\frac q2 -\frac{(q-2)(s+1)}{4s}}\left(\int_{\R^2} |\partial_x u|^2 \, dxdy \right)^{\frac{q-2}{4}}\left(\int_{\R^2} |D_y^s  u|^2 \, dx \right)^{\frac{q-2}{4s}}.
	\end{align}
\end{lemma}
Note that all the exponents on the right hand side of \eqref{gn} are positive\footnote{While the exponent for $\|u\|_2$ becomes zero as $q=p_s.$}   for $2<q<p_s$. Raising \eqref{gn} to  power $\frac{2}{q}$, we obtain that
\begin{equation*}
	\left(	\int_{\R^2}|u|^q \,dx dy\right)^{\f{2}{q}}\leq C \left(\int_{\R^2}|u|^2 \, dxdy \right)^{1 -\frac{2(q-2)(s+1)}{4sq}}\left(\int_{\R^2} |\partial_x u|^2 \, dxdy \right)^{\frac{2(q-2)}{q}}\left(\int_{\R^2} |D_y^s  u|^2 \, dx \right)^{\frac{2(q-2)}{4sq}}.
\end{equation*}
As a consequence, for all $q: 2<q<p_s$, we can apply the Young's inequality to obtain the following estimate,
\begin{equation}
	\label{30}
	\left(	\int_{\R^2}|u|^q \,dx dy\right)^{\f{2}{q}}\leq  C \int_{\rtwo} \left(|\partial_x u|^2+|D_y^s u|^2 + |u|^2 \right) dx dy,
\end{equation}
which will be useful in the sequel. Note that this is equivalent to the Sobolev type embedding $H^{1,s}(\R^d)\hookrightarrow L^q(\R^d)$, which holds under the assumption $s>0$ and $2<q\leq p_s$.
This is also observed in \cite[Remark 2.2]{E}.
\subsection{Pohoz\={a}ev identities and consequences}
We have the following lemma, which prescribes necessary conditions (better known as Pohozaev identitites) for appropriately localized solutions of \eqref{20}.
\begin{lemma}
	\label{poh}
	Let $u\in H^{1,s}(\rtwo)\cap L^p(\rtwo)$ is a solution to \eqref{20}. Then
	\begin{eqnarray}
		\label{40}
	& & 	s \left(\int_{\rtwo}  |\partial_ x u|^2 \, dxdy +\int_{\rtwo} |D_y^s u|^2 \, dxdy\right) = \frac {(s+1)(p-2)}{2p} \int_{\R^2} |u|^p \, dxdy
\end{eqnarray}
and
	\begin{eqnarray} \label{45}
& & 	\int_{\rtwo} |\partial_ x u|^2 \, dxdy +\int_{\rtwo}|D_y^s  u|^2 \, dxdy + \om  \int_{\rtwo}|u|^2 \, dx =  \int_{\rtwo} |u|^p \, dxdy.
	\end{eqnarray}
\end{lemma}
Before we proceed with the proof of Lemma \ref{poh}, let us derive a few more relations.
\subsubsection{Necessary conditions: Proof of Proposition \ref{prop:12}}
First, from \eqref{40} it is clear that $p>2$. Combining \eqref{40} and \eqref{45}, we obtain a relation between $\|u\|_2$ and $\|u\|_p$, which reads
$$
\om \int_{\R^2}|u|^2 \, dxdy=\left(1-\frac{(s+1)(p-2)}{2ps}\right) \int_{\R^2}|u|^p \, dxdy.
$$
It follows that $1-\frac{(s+1)(p-2)}{2ps}>0$, which is equivalent to $p<p_s$. Clearly, this shows that the condition $2<p<p_s$ is necessary, as is one of the claims of Proposition \ref{prop:12}.

We now show that $\om>0$ is necessary as well. Starting with \eqref{20},
we integrate in the $y$ variable. Upon the introduction of
$$
\psi(x):=\int_{-\infty}^\infty \phi(x,y) dy, \quad V(x):=\int_{-\infty}^\infty \phi^{p-1}(x,y) dy,
$$
we have that they satisfy
\begin{equation}
	\label{c:12}
-\psi''+\om \psi = V.
\end{equation}
 Note that by the assumptions for $\phi$, namely that it is equal to its Steiner rearrangement, it then follows that $\psi$ and $V$ are bell-shaped, i.e. positive symmetric and decreasing in $(0, \infty)$. It is clear that $\om=0$ is immediately problematic in \eqref{c:12}, since
 $$
 0=-\int_{\rone}  \psi'' \, dx= \int_{\rone}  V(x) \,dx>0,
 $$
 a contradiction.  For $\om=-\si^2, \si>0$, we see that \eqref{c:12} implies in particular that
 $$
 \int_{\rone} V(x) \sin(\si x) dx=0.
 $$
 By symmetry, it must be that
$$
\int_{0}^\infty V(x) \sin(\si x) dx=0,
$$
but we claim that this is contradictory. Indeed, by the positivity and decay of $V$, we have that
 $$
 \int_{\f{2k \pi}{\si}}^{\f{(2k+1)\pi}{\si}}  V(x) \sin(\si x) dx>-\int_{\f{(2k+1)  \pi}{\si}}^{\f{(2k+2)\pi}{\si}}  V(x) \sin(\si x) dx, \quad k=0,1,2,\ldots,
 $$
 whence
 $$
 0=\int_{0}^\infty V(x) \sin(\si x) dx = \sum_{k=0}^\infty \left( \int_{\f{2k \pi}{\si}}^{\f{(2k+1)\pi}{\si}}  V(x) \sin(\si x) dx+\int_{\f{(2k+1)  \pi}{\si}}^{\f{(2k+2)\pi}{\si}}  V(x) \sin(\si x) dx\right)>0.
 $$
To complete the proof, it now remains to complete the proof of the Pohozaev identities.

\subsubsection{Proof of Lemma \ref{poh}}

First, it follows from Theorem \ref{theo:10} that $x \partial_x u, y \partial_y u \in L^2(\R^2)$. Multiplying \eqref{10} by $y \partial_y$ and integrating on $\R^2$, we get that
	\begin{align*}
		-\int_{\R^2} y \partial_{xx} u \partial_y u\, dxdy + \int_{\R^N} (-\p_{yy})^s u\left(y \partial_y u\right) \, dxdy+  \om \int_{\R^2}  y u\partial_y u\, dxdy= \int_{\R^2} y|u|^{p-2} u\partial_y u  \, dxdy.
	\end{align*}
	Next we compute each term in the identity above. In light of the divergence theorem, we derive that
	$$
	-\int_{\R^2} y \partial_{xx} u \partial_y u\, dxdy=\int_{\R^2} y \partial_{x} u \partial_{xy} u\, dxdy=\frac 12 \int_{\R^2} y \partial_{y} |\partial_{x} u|^2\, dxdy=-\frac 12 \int_{\R^2} |\partial_ x u|^2 \, dxdy,
	$$
	$$
	\om \int_{\R^2}  y u\partial_y u\, dxdy=\frac{\om}{2} \int_{\R^2}  y \partial_y |u|^2\, dxdy=-\frac{\om}{2} \int_{\R^2} |u|^2 \, dxdy
	$$
	and
	$$
	\int_{\R^2} y|u|^{p-2} u\partial_y u  \, dxdy=\frac 1p \int_{\R^2} y\partial_y |u|^p  \, dxdy =-\frac 1p \int_{\R^2} |u|^p \, dxdy.
	$$
	Note that $[(-\p_{yy})^s,y \partial_y]=2s (-\p_{yy})^s$ and
	\begin{align*} 
		(-\p_{yy})^s(y \partial_y u)=2s (-\p_{yy})^s u + y \partial_y (-\p_{yy})^s u.
	\end{align*}
	Moreover, there holds that
	\begin{align*}
		\int_{\R^2} (-\p_{yy})^s  u\left(y \partial_y u\right) \, dxdy&= \int_{\R^2} u (-\p_{yy})^s \left(y \partial_y u\right) \, dxdy \\
		&=2s \int_{\R^2} u (-\p_{yy})^su \, dxdy +\int_{\R^2} y u \partial_y (-\p_{yy})^s u\, dxdy \\
		&=(2s-1) \int_{\R^2} u (-\p_{yy})^s u \, dxdy -\int_{\R^2} u (-\p_{yy})^s \left(y \partial_y u\right) \, dxdy.
	\end{align*}
	This then leads to
	$$
	\int_{\R^2} (-\p_{yy})^s  u\left(y \partial_y u\right) \, dxdy=\frac{2s-1}{2}\int_{\R^2} u (-\p_{yy})^s u \, dxdy =\frac{2s-1}{2}\int_{\R^2} |(-\Delta )^{\frac s 2}_y u|^2 \, dxdy.
	$$
	Therefore, we obtain that
	\begin{align} \label{ph2}
		\frac 12 \int_{\R^2} |\partial_ x u|^2 \, dxdy +\frac{1-2s}{2}\int_{\R^2} |D_y^s u|^2 \, dxdy +\frac {\om}{2} \int_{\R^2} |u|^2 \, dxdy=\frac 1p \int_{\R^2} |u|^p \, dxdy.
	\end{align}
	On the other hand, multiplying \eqref{20} by $x \partial_x u$ and integrating on $\R^2$, we have that
	\begin{align*} 
		-\int_{\R^2} x \partial_{xx} u \partial_x u\, dxdy + \int_{\R^N} (-\p_{yy})^s  u\left(x \partial_x u\right) \, dxdy+  \om \int_{\R^2}  x u\partial_x u\, dxdy= \int_{\R^2} x|u|^{p-2} u\partial_x u  \, dxdy.
	\end{align*}
	Now we calculate each term in the identity above. Applying the fact that $ \p_{yy}$ commutes with  $x \partial_x u$ and the divergence theorem, we are able to similarly derive that
	$$
	-\int_{\R^2} x \partial_{xx} u \partial_x u\, dxdy= \frac 12 \int_{\R^2} |\partial_x u|^2 \, dxdy, \quad  \int_{\R^N} (-\p_{yy})^s  u\left(x \partial_x u\right) \, dxdy=-\frac 12 \int_{\R^2} |D_y^s u|^2 \, dxdy
	$$
	and
	$$
	\om \int_{\R^2}  x u\partial_x u\, dxdy=-\frac{\om}{2} \int_{\R^2}|u|^2 \,dxdy, \quad \int_{\R^2}  x u\partial_x |u|^{p-2} u\, dxdy = -\frac{1}{p} \int_{\R^2}|u|^p \, dxdy.
	$$
	As a consequence, we get that
	$$
	-\frac 12 \int_{\R^2} |\partial_ x u|^2 \, dxdy +\frac{1}{2}\int_{\R^2} |D_y^s u|^2 \, dxdy +\frac {\om}{2} \int_{\R^2} |u|^2 \, dxdy=\frac 1p \int_{\R^2} |u|^p \, dxdy.
	$$
	This along with \eqref{ph2} results in
	\begin{align*} 
		\frac{1-s}{2} \int_{\R^2} |\partial_ x u|^2 \, dxdy +\frac{1-s}{2}\int_{\R^2} |D_y^s u|^2 \, dxdy +\frac {\om(s+1)}{2} \int_{\R^2} |u|^2 \, dxdy=\frac {s+1}{p} \int_{\R^2} |u|^p \, dxdy.
	\end{align*}
	In addition, by multiplying \eqref{20} by $u$ and integrating on $\R^2$, we conclude that
	$$
	\int_{\R^2} |\partial_ x u|^2 \, dxdy +\int_{\R^2} |D_y^s  u|^2 \, dxdy +\om \int_{\R^2} |u|^2 \, dxdy=\int_{\R^2} |u|^p \, dxdy.
	$$
	As a result, we now obtain that
	$$
	s \int_{\R^2} |\partial_ x u|^2 \, dxdy +s\int_{\R^2} |D_y^s u|^2 \, dxdy =\frac {(s+1)(p-2)}{2p} \int_{\R^2} |u|^p \, dxdy.
	$$
	The proof of Lemma \ref{poh} is thus completed.

\section{Variational construction and stability of the waves}
\label{sec:4}
\subsection{Variational setup}

Our goal in this section is to construct the waves satisfying \eqref{20}.
For this, we introduce the Weinstein functional
$$
J[u]:= \frac{\int_{\rtwo} \left(|\partial_x u|^2+|D_y^s u|^2 + \om |u|^2 \right) dx dy}{\left(\int_{\rtwo} |u|^p \,dx dy\right)^{\f{2}{p}}}.
$$
Define
\begin{align} \label{min}
J_{\om}:=\inf_{\{u\in H^{1,s}(\R^2): u\neq 0\}} J[u].
\end{align}
Obviously, as a consequence of \eqref{30}, we see that $J_{\om}>0$.
\begin{proposition}
	\label{prop:10}
	Let $0<s<1$, $\om>0$ and $2<p<p_s$. Then the variational problem \eqref{min} has an axially symmetric solution $\Phi>0$.  In addition, $\Phi$ obeys the following equation,
	\begin{equation}
		\label{c:10}
		-\p_{xx} \Phi+ (-\p_{yy})^s \Phi+\om \Phi = C  \Phi^{p-1},
	\end{equation}
for some constant $C=C(\om, \Phi)>0$, which is given by
\begin{equation}
	\label{c:20}
	C=\f{\int_{\R^2}\left(|\partial_x \Phi|^2+|D_y^s \Phi|^2 + \om |\Phi|^2 \right) \, dxdy}{\|\Phi\|_{p}^{p}}.
\end{equation}
\end{proposition}
\begin{rmk}
By standard elliptic theories along with bootstrap arguments and since $\Phi^{p-1}\in L^2(\rtwo)$, one can upgrade the smoothness to at least $\Phi\in H^{2,2s}(\rtwo)$. 
\end{rmk}
\begin{proof}
Let $\{u_n\} \subset H^{1,s}(\R^2)$ satisfying $\|u_n\|_p=1$ be a minimizing sequence of \eqref{min}, i.e, $J[u_n]=J_{\omega}+o_n(1)$.
Note that $\|u_n^{\sharp}\|_p=\|u_n^{\ast}\|_p=1$ and 
$$
J(u_n^{\sharp})\leq J(u_n), \quad J(u_n^{\ast})\leq J(u_n),
$$
where $u^\sharp$ denotes the axial symmetrization with respect to the $x$-axis, and $u^{\ast}$ denotes the symmetrization with respect to the $y$-axis. Define $\Phi_n:=(u_n^\sharp)^{\ast}$. Then $\Phi_n \in H^{1,s}_a(\R^2)$ be a minimizing sequence to \eqref{min} satisfying $\|\Phi_n\|_p=1$. Further, we have that $\{\Phi_n\}$ is bounded in $H^{1,s}(\R^2)$. In view of \cite[Page 3]{EE}, then there exists $\Phi \in H^{1,s}(\R^2)$ such that $\Phi_n \rightharpoonup \Phi$ in $H^{1,s}(\R^2)$ and $\Phi_n \to \Phi$ in $L^p(\R^2)$ as $n \to \infty$. This implies that $\Phi$ is a nonnegative minimizer to \eqref{min} .

We now aim to establish that $\Phi$ obeys the Euler-Lagrange equation \eqref{c:10}. This is standard by exploiting the fact that $\Phi$ is a minimizer to \eqref{min}.  To this end,
	let $\eps>0$ and $h\in \cs$ be a test function. Then $J[\Phi+\eps h]\geq J[\Phi]$.
	But expanding up to order $O(\eps^2)$ yields that
	\begin{align*}
\int_{\R^2}  \left(|\partial_x (\Phi+\eps h)|^2+|D_y^s (\Phi+\eps h)|^2 + \om |\Phi+\eps h|^2 \right) \,dxdy&= \int_{\R^2}  \left(|\partial_x \Phi|^2+|D_y^s \Phi|^2 + \om |\Phi|^2 \right) \, dxdy\\
& \quad + 2 \eps\dpr{(\left(-\partial_{xx}\right)+D_y^{2s}+\om)\Phi}{h}+O(\eps^2)
\end{align*}
and
\begin{align*}
\| \Phi+\eps h\|_{p}^{-2} = \|\Phi\|_{p}^{-2} - 2\eps \|\Phi\|_{p}^{-p-2}\dpr{\Phi^{p-1}}{h}+O(\eps^2).
	\end{align*}
Therefore, we have that
\begin{align*}
J[\Phi+\eps h] &=  J[\Phi]+  2 \eps \|\Phi\|_{p}^{-2} \left(\dpr{((-\partial_{xx})+D_y^{2s}+\om)\Phi}{h} - C(\om, \Phi)\dpr{\Phi^{p-1}}{h}  \right)+ O(\eps^2),
	\end{align*}
where $C(\om, \Phi)$ is defined by \eqref{c:20}. Since $\eps$ and $h$ are arbitrary, then we obtain that $\Phi$ is a solution of  \eqref{c:10}. By the maximum principle, 
we further know that $\Phi>0$. This completes the proof.

\end{proof}

Next we establish some useful properties for the corresponding linearized operators $L_\pm$, where $L_\pm$ are defined by
	\begin{align*}
	L_+ := -\p_{xx}  +(-\p_{yy})^s +\om - (p-1) C(\om,\Phi) \Phi^{p-2}, \quad
	L_- := -\p_{xx}  +(-\p_{yy})^s +\om -  C(\om,\Phi)\Phi^{p-2}.
\end{align*}
As both of these are symmetric operators, then we consider them both as Friedrich extended to self-adjoint operators with the standard domain $D(L_\pm)=H^{2,2s}(\R^2)$. By Weyls's theorem, we have that the essential spectrum of both of these operators fulfill that
$$
\si_{ess.}(L_+)=\si_{ess.}(L_-)=[\om, \infty).
$$
Our next goal is to address the structure of the  spectrum of $L_\pm$ outside of its essential spectra. These are indeed only point spectrum, i.e. eigenvalues, due to the fact that their $L_\pm$ are both relatively compact perturbations of the constant coefficient operator,
$L=-\p_{xx}  +(-\p_{yy})^s +\om$. Also we introduce the Morse index notation for self-adjoint operators $S$, possessing at most finitely many negative eigenvalues,
$$
n(S)=\#\{\la<0: \la\in \si_{p.p.}(S)\},
$$
where we count eigenvalues with their respective algebraic multiplicities.
\begin{proposition}
	\label{prop:20}
Let $0<s<1$, $\om>0$ and $2<p<p_s$. Then linearized operators $L_\pm$ enjoy the properties
\begin{itemize}
	\item $L_+$ has exactly one negative eigenvalue and $Ker(L_+)$ is finite dimensional.
	\item $L_-\geq 0$. In fact, $Ker(L_-)=span[\Phi]$, while $L_-|_{\{\Phi\}^\perp}\geq \de>0$, for some $\de>0$.
\end{itemize}
\end{proposition}
\begin{proof}
	The statement about $L_+$ follows from the minimization property of $\Phi$. Indeed, taking a test function as in the proof of Proposition \ref{prop:10}, $h: h\perp \Phi^{p-1}$, we obtain that
		\begin{align*}
\int  \left(|\partial_x (\Phi+\eps h)|^2+|D_y^s (\Phi+\eps h)|^2 + \om |\Phi+\eps h|^2 \right) \,dxdy	 &=  \int  \left(|\partial_x \Phi|^2+|D_y^s \Phi|^2 + \om |\Phi|^2 \right) \, dxdy \\
		&\quad + \eps^2  (\dpr{(-\partial_{xx}) +D_y^{2s}+\om)h}{h} +o(\eps^2)
		\end{align*}
		and
		\begin{align*}
		\| \Phi+\eps h\|_{p}^{-2} = \|\Phi\|_{p}^{-2} - (p-1) \eps^2  \|\Phi\|_{p}^{-p-2}\dpr{\Phi^{p-2}}{h^2}+o(\eps^2).
	\end{align*}
Therefore, we derive taht
\begin{align*}
	J[\Phi+\eps h] =  J[\Phi]+  \eps^2 \|\Phi\|_{p}^{-2} \left(\dpr{(-\partial_{xx}) +D_y^{2s}+\om)h}{h} -(p-1)
	C(\om, \Phi) \dpr{\Phi^{p-2}}{h^2}  \right)  + o(\eps^2).
\end{align*}
Taking into account that the second variation of the minimizer needs to be non-negative, and the formula \eqref{c:20}, we conclude that $\dpr{L_+ h}{h}\geq 0$ for all $h: h\perp \Phi^{p-1}$. By the min-max realization of the eigenvalues of self-adjoint operators, this implies that $L_+$ has at most one negative eigenvalue, i.e. $n(L_+)\leq 1$.  Since on the other hand,
$$
\dpr{L_+\Phi}{\Phi}=-(p-2) C(\om,\Phi) \int \Phi^p dx<0,
$$
it then follows, again by the min-max realization of the eigenvalues, that $L_+$ does have a negative eigenvalue, which is simple as well. Clearly, $L_+$ has kernel, in fact a differentiation in $x$ and $y$ variables in \eqref{c:10} reveals that
$$
L_+[\p_x \Phi]=0=L_+[\p_y \Phi].
$$
Finally, $Ker(L_+)$ is finite dimensional,  as one can interpret the problem for the elements in $Ker(L_+)$ as the eigenvalue problem for   the following compact operator $T$ on $L^2(\R^2)$,
$$T f:= ((-\partial_{xx})+D_y^{2s}+\om)^{-1} \left[\Phi^{p-2} f\right], \quad f \in L^2(\R^2).$$

Regarding $L_-$, observe that since $\Phi>0$, we have that $L_-\geq L_+$, whence $n(L_-)\leq 1$. In fact,   $L_-[\Phi]=0$ is nothing but an alternative rewrite of \eqref{c:10}, whence $0\in \si_{p.p.}(L_-)$. As $\Phi> 0$, and $L_-$ enjoys the Perron-Frobenius property\footnote{This can be shown as a consequence of the fact that $e^{-tL_-}$ is positivity preserving semigroup, which is standard for fractional Schr\"odinger operators with $0<s\leq 1.$} (i.e. the bottom of the spectrum consists of a  simple e-value, corresponding to a strictly positive eigenfunction), it then follows that $L_-\geq 0$, and $L_-|_{\{\Phi\}\}^\perp}\geq \de>0$, for some $\de>0$. Thus the proof is completed.
\end{proof}

\subsection{Stability of the waves $\phi_\om$}
We start with the construction of the waves $\phi_\om$. Most of the work has been done in Section \ref{sec:2}, but we provide a few extra details herein.
\subsubsection{Construction of the wave $\phi$}
We are now ready to finalize our construction of the waves $\phi$. 
Let $\Phi$ be a minimizer to \eqref{min} and define
$$
\phi=\phi_\om:= C^{\f{1}{p-2}} \Phi,
$$
then $\phi$ is a solution of the required elliptic profile equation \eqref{20}, where $C>0$ is given by \eqref{c:20}. Moreover, one can easily translate the properties of the linearized operators $L_\pm$ to the properties of $\cl_{\pm, s}$, which appear in the eigenvalue problem \eqref{90}.  Hereinafter, we shall use $\cl_{\pm}$ to denote $\cl_{\pm, s}$ for simplicity. We collect the pertinent information in the following proposition.
\begin{proposition}
	\label{prop:30}
	Let $0<s<1$, $\om>0$ and $2<p<p_s$. Then the function defined via
	$$
	\phi_\om =\left(\f{\int_{\R^2}\left(|\partial_x \Phi|^2+|D_y^s \Phi|^2 + \om |\Phi|^2 \right) \,dxdy}{\|\Phi\|_{p}^{p}}\right)^{\f{1}{p-2}} \Phi,
	$$
	where $\Phi$ is the function produced in Proposition \ref{prop:10} and
	satisfies \eqref{20}. Also, the operators $\cl_\pm$ obey the following:
	\begin{itemize}
		\item $\cl_+$ has exactly one negative eigenvalue, so $n(\cl_+)=1$,  $Ker(\cl_+)$ is finite dimensional.
		\item $\cl_-\geq 0$, $Ker(\cl_-)=span[\phi]$ and  $\cl_-|_{\{\phi\}^\perp}\geq \de>0$, for some $\de>0$.
	\end{itemize}
\end{proposition}
An interesting property that can be gleaned right away is that a simple scaling can describe the dependence $\om\to \phi_\om$. Indeed, it is clear that $\phi_\om$ can be generated from $\phi_1$ as follows
\begin{equation}
	\label{d:10}
	\phi_\om(x,y)=\om^{\f{1}{p-2}} \phi_1(\om^{\f{1}{2}} x,\om^{\f{1}{2s}} y).
\end{equation}
\subsubsection{Hamiltonian instability index}
In this subsection, we discuss the stability of the waves, whose construction is given in    \eqref{d:10}. Specifically, we employ the theory of the instability Hamiltonian index.  For the eigenvalue problem in the form
\begin{equation}
	\label{s4.1}
	\mathcal{J} \mathcal{H} V=\lambda V,
\end{equation}
where we assume that $\ch=\ch^*, \cj^*=-\cj$ has $dim(Ker(\mathcal{H})<\infty$, and also a finite number of negative eigenvalues $n(\mathcal{H})$.

Let $k_r$  be  the number of positive eigenvalues of the spectral problem \eqref{s4.1} (i.e. the number of real instabilities or real modes), $k_c$ be the number of quadruplets of eigenvalues with non-zero real and imaginary parts, and $k_i^-$, the number of pairs of purely imaginary eigenvalues with negative Krein-signature.  For a simple pair of imaginary eigenvalues $\pm i \mu$ and $\mu\neq 0$, and the corresponding eigenvector
$\vec{z} = \left(\begin{array}{c}
	z_1  \\ z_2
\end{array}\right) $, the Krein index is
$
sgn(\dpr{ \mathcal{H}  \vec{z}}{ \vec{z}}).
$

Also of importance in this theory is a finite dimensional matrix $\cd$, which is obtained from the adjoint eigenvectors for \eqref{s4.1}. More specifically, consider  the generalized kernel of $\mathcal{J}\mathcal{H}$,
$$
gKer(\mathcal{J}\mathcal{H})=span[(Ker(\mathcal{J}\mathcal{H}))^l, l=1, 2, \ldots]  .
$$
Assume that $dim(gKer(\cj\ch))<\infty$. Select a basis in $gKer(\cj\ch)\ominus Ker(\ch)=span[\eta_j, j=1, \ldots, N]$.
Then $\cd\in M_{N\times N}$ is defined via
$$
\cd:=\{\cd_{i j}\}_{i,j=1}^N: \cd_{i j}=\dpr{\ch \eta_i}{\eta_j} .
$$
Then, following \cite{KKS1, KKS2}, we have the following formula, relating the number of
``instabilities'' or Hamiltonian  index of the eigenvalue problem \eqref{s4.1} and the Morse indices of the operators $\ch$ and $\cd$
\begin{equation}
	\label{b:20}
	k_{Ham}:=k_r+2 k_c+2k_i^-=n(\ch)-n(\cd).
\end{equation}
As a corollary, if $n(\ch)=1$, we see that $k_c=k_i^-=0$, hence $k_{Ham.}=n(\ch)-n(\cd)=1-n(\cd)$. It follows that the stability is equivalent to $n(\cd)=1$, while instability occurs if $n(\cd)=0$.
\subsubsection{Determining the stability of the waves $\phi$}
We apply the theory developed in the previous section (specifically the formula \eqref{b:20}) to our eigenvalue problem \eqref{90}. As $\cj= \left(\begin{array}{cc}
	0 & -1 \\ 1 & 0
\end{array}\right)$, while $\cl=  \left(\begin{array}{cc}
	\cl_+ & 0 \\ 0& \cl_-
\end{array}\right) $, we see that $n(\cl)=n(\cl_+)+n(\cl_-)=1$, according to Proposition \ref{prop:30}.

It remains to identify  $gKer(\cl)$ and consequently the matrix $\cd$. First
$$
Ker(\cl)=\left\{\left(\begin{array}{c} Ker(\cl_+) \\ 0   \end{array}\right) , \left(\begin{array}{c}  0 \\ Ker(\cl_-)    \end{array}\right) \right\}
$$
and both $Ker(\cl_\pm)$ are finite dimensional. In addition, based on the formula \eqref{d:10}, we can take a derivative in $\om$ in the profile equation \eqref{20} to obtain
\begin{equation}
	\label{d:20}
	\cl_+[\p_\om \phi_\om]=-\phi_\om.
\end{equation}
We see that the adjoint eigenvectors corresponding to $\left(\begin{array}{c} Ker(\cl_+) \\ 0   \end{array}\right)$ are in the form
$$
\cj \cl \left(\begin{array}{c} 0 \\ f_j   \end{array}\right) = \left(\begin{array}{c} g_j  \\ 0 \end{array}\right), \quad g_j\in Ker(\cl_+).
$$
This has the  the unique solution $f_j=-\cl_-^{-1} g_j \in Ker(\cl_+)^\perp$, since by \eqref{d:20} we have that $\phi_\om \perp Ker(\cl_+)$.  We attempt for another set of adjoints, that is we look to solve the equation
$$
\cj \cl \vec{z}=\left(\begin{array}{c} 0 \\ f_j   \end{array}\right),
$$
which does not have solutions, as it requires to resolve
\begin{equation}
	\label{d:40}
	\cl_+ z_1=f_j=-\cl_-^{-1} g_j.
\end{equation}
The last equation however does not have solutions, since a testing with $g_j$ yields
$$
0= \dpr{\ z_1}{\cl_+ g_j}=\dpr{\cl_+ z_1}{g_j}=-\dpr{\cl_-^{-1} g_j}{g_j}<0,
$$
as $g_j\in Ker(\cl_+) \perp \phi$ and $\cl_-|_{\{\phi\}^\perp}>0$, and hence $\cl_-^{-1}|_{\{\phi\}^\perp}>0$.

The adjoints arising from the other piece of the kernel, we have
$$
\cj \cl \vec{z}=\left(\begin{array}{c} 0 \\ \phi  \end{array}\right),
$$
which yields $\vec{z}=\left(\begin{array}{c}   \cl_+^{-1} \phi  \\ 0 \end{array}\right)$. Attempting for a second adjoint is again  like solving
\begin{equation}
	\label{d:50}
	\cj \cl \vec{z}=\left(\begin{array}{c}   \cl_+^{-1} \phi  \\ 0 \end{array}\right).
\end{equation}
Then \eqref{d:50} is reduced to $\cl_- z_2= -\cl_+^{-1} \phi$. By Fredholm alternative, this is sovable, precisely when $\dpr{\cl_+^{-1} \phi}{\phi}=0$. Thus, we have proved the following proposition.
\begin{proposition}
	\label{prop:65}
	Let $0<s<1$, $\om>0$ and $2<p<p_s$. Then, assuming that $\dpr{\cl_+^{-1} \phi}{\phi}\neq 0$, we have that
	$$
	gKer(\cj \cl)=span\{\left(\begin{array}{c}   \cl_+^{-1} \phi  \\ 0 \end{array}\right)\cup Ker(\cl)\}.
	$$
\end{proposition}

\begin{proof} [Proof of Theorem \ref{theo:20}]
Applying the theory developed in the previous section, we obtain that for $\dpr{\cl_+^{-1} \phi}{\phi}\neq 0$, $\cd$ is a matrix of one element $\cd_{1 1}$, which is
$$
\cd_{1 1}=\dpr{\cl_+^{-1}\phi}{\phi}.
$$
This quantity may be computed with the help of the relation \eqref{d:20}. Indeed, according to it, $\cl_+^{-1} \phi_\om = -\p_\om \phi_\om$, so that
\begin{align*}
\cd_{1 1}=-\dpr{\p_\om \phi_\om}{\phi_\om}=-\f{1}{2}\p_\om \|\phi_\om\|_{L^2}^2
&= -const. \p_\om [\om^{\f{2}{p-2}-\f{1}{2}-\f{1}{2s}}]\\
&=const. \left(\f{1}{2}+\f{1}{2s}-\f{2}{p-2}\right) [\om^{\f{2}{p-2}-\f{3}{2}-\f{1}{2s}}].
\end{align*}
As stability is equivalent to $\cd_{11}<0$, we arrive at the condition
$
\f{1}{2}+\f{1}{2s}-\f{2}{p-2}<0,
$
or equivalently the mass-subcritical range
$$
2<p<\f{6s+2}{s+1}=p_m.
$$
As a result, there holds that $\cd_{1,1}>0$ in the complementary range $p_m<p<p_s$ and $\cd_{1,1}=0$ for $p=p_m$. If $\cd_{1,1}>0$, then there exist a pair (one positive and one negative) of eigenvalues to \eqref{90}. By the continuity of the spectrum with respect to parameters, if $\cd_{1,1}=0$, then the pair transitions through the origin to become a pair of purely imaginary eigenvalues. This completes Theorem \ref{theo:20}.
\end{proof}

\section{Uniqueness of ground states}
\label{unique}

In the section, we are going to discuss the uniqueness of ground states to \eqref{20}. Our aim is to prove Theorems \ref{uniqueness} and \ref{nd}.

Indeed, to establish Theorem \ref{uniqueness}, by scaling techniques $\phi_\omega(x,y)=\omega^{\frac{1}{p-2}}\phi(\omega^{\frac12}x,\omega^{\frac{1}{2s}}y)$, we only need to show the uniqueness of the ground state to the equation
\begin{align} \label{201}
	-\partial_{xx} u + (-\p_{yy})^s  u + u=|u|^{p-2}u.
\end{align}
For this, we shall closely follow strategies developed in \cite{FL, FLS}, recently extended in \cite{HS}. To begin with, we define
$$
X_p:=\left\{u \in L^2(\R^2) \cap L^p(\R^2) : u \,\, \mbox{is axially symmetric and real-valued} \right\}
$$
equipped with the norm
$$
\|u\|_{X_p}:=\|u\|_2 + \|u\|_p.
$$

\begin{lemma} \label{regularity}
	Let $0<s<1$, $2<p<p_s$ and $u \in L^2(\R^2) \cap L^p(\R^2)$ be a solution to \eqref{201}. Then $u \in H^{1, s}(\R^2)$.
\end{lemma}
\begin{proof}
	Since $u \in L^2(\R^2) \cap L^p(\R^2)$ solves \eqref{201}, then
	$$
	u= \frac{1}{-\partial_{xx}  + (-\p_{yy})^s  + 1} |u|^{p-2} u.
	$$
	In light of Plancherel's identity, then
	\begin{align*}
		\left\| \left(\partial_{x}  + (-\partial_{yy})^{\frac{s}{2}} \right) u \right\|_2=\left\| \frac{\partial_{x}  + (-\partial_{yy})^{\frac{s}{2}}}{-\partial_{xx}  + (-\p_{yy})^s  + 1} |u|^{p-2} u \right\|_2 \lesssim \left\| \frac{1}{\left(-\partial_{xx} \right)^{\frac 12} + (-\p_{yy})^{\frac s 2}  + 1} |u|^{p-2} u \right\|_2.
	\end{align*}
	In order to further estimate the term in the right hand side of the inequality above, we introduce the fundamental solution $\mathcal{K}$ to the equation
	$$
	\left(-\partial_{xx}\right)^{\frac 12} u + (-\p_{yy})^{\frac s 2}  u + u=0.
	$$
	It is not hard to check that
	$$
	\mathcal{K}(x, y)=\int_0^{+\infty} e^{-t} \mathcal{H}(x, y, t) \, dt,
	$$
	where
	$$
	\mathcal{H}(x,y,t):=\int_{\R} \int_{\R} e^{-2\pi \textnormal{i} (x ,y) \cdot (\xi_1, \xi_2)-t \left(|\xi_1| +|\xi_2|^s\right)} \, d\xi_1d\xi_2:=\mathcal{H}_1(x, t)\mathcal{H}_s(y, t)
	$$
	and
	$$
	\mathcal{H}_1(x, t):=\int_{\R} e^{-2\pi \textnormal{i} x \xi_1 -t |\xi_1|} \, d\xi_1, \quad \mathcal{H}_s(y, t):=\int_{\R} e^{-2\pi \textnormal{i}y \xi_2 -t |\xi_2|^s} \, d\xi_2.
	$$
	By straightforward calculations, we get that
	$$
	\mathcal{H}_1(x, t)=\frac{2t}{t^2+4 \pi^2 |x|^2}.
	$$
	Furthermore, it follows from $(A 4)$ in \cite[Appendix A]{FQT} that
	$$
	0 <\mathcal{H}_{s}(y ,t) \lesssim \min \left\{t^{-\frac{1}{s}}, t|y|^{-1-s}\right\}.
	$$
	Therefore, we are able to derive that, for any $q>1$,
	\begin{align*}
		\|\mathcal{H}(\cdot, t)\|_q =\|\mathcal{H}_1(\cdot,t) \mathcal{H}_s(\cdot,t) \|_q
		&\lesssim t^{1-\frac{1}{s}}  \left(\int_{\R} \int_{|y| \leq t^{\frac{1}{s}}} \frac{1}{\left(t^2+4 \pi^2 |x|^2\right)^q} \, dxdy\right)^{\frac 1q} \\
		& \quad + t^{2}\left(\int_{\R} \int_{|y|>t^{\frac{1}{s}}} \frac{1}{\left(t^2+4 \pi^2 |x|^2\right)^q}   |y|^{-q-sq} \, dxdy\right)^{\frac 1q} \\
		& \lesssim t^{-\left(1-\frac 1q \right)\left(\frac 12+\frac{1}{s}\right)}.
	\end{align*}
	This leads to
	$$
	\|\mathcal{K}\|_q \leq \int_0^{+\infty} e^{-t} \|\mathcal{H}(\cdot, t)\|_q \, dt \lesssim \int_0^{+\infty} e^{-t} t^{-\left(1-\frac 1q \right)\left(\frac 12+\frac{1}{s}\right)} \, dt<+\infty,
	$$
	where $q>1$ satisfies the condition
	\begin{align} \label{cqq}
		\left(1-\frac 1q \right)\left(\frac 12+\frac{1}{s} \right)<1.
	\end{align}
	Coming back to the previous estimate and using Young's inequality under the condition \eqref{cqq}, we then get that
	\begin{align*}
		\left\| \frac{1}{\left(-\partial_{xx} \right)^{\frac 12} + (-\p_{yy})^{\frac s 2}  + 1} |u|^{p-2} u \right\|_2 =\left\| \mathcal{K}  \ast |u|^{p-2} u \right\|_2 \lesssim  \|u\|_p^{p-1}<+\infty,
	\end{align*}
	where
	$$
	1-\frac 1q=\frac 12 -\frac  1p<\frac{s}{1+s}, \quad 2<p<p_s,
	$$
	and \eqref{cqq} is satisfied. This implies that $u \in H^{1,s}(\R^2)$ and the proof is completed.
\end{proof}

\begin{lemma} \label{bri}
	Let $0<s_0<1$ and $2<p<p_{s_0}$. Suppose that $u_0 \in X_p$ solves \eqref{201} with $s=s_0$ and the linearized operator
	$$
	\mathcal{L}_{+, s_0}:=-\partial_{xx}  + (-\partial_{yy})^{s_0}   + 1-(p-1)|u_0|^{p-2}
	$$
	has trivial kernel on $L^2_{a}(\R^2)$. Then there exist $\delta_0>0$ and a map $u \in C^1(I; X_p)$ with $I=[s_0, s_0+ \delta_0)$ such that
	\begin{itemize}
		\item[$(\textnormal{i})$] $u_s$ solves \eqref{201} for any $s \in I$, where $u_s:=u(s)$ for $s \in I$.
		\item [$(\textnormal{ii})$] There exists $\eps>0$ such that $u_s$ is the unique solution to \eqref{201} for $s \in I$ in the neighborhood
		$$
		\left\{u \in X_p : \|u-u_0\|_{X_p} < \eps \right\},
		$$
		where $u_0=u_{s_0}$.
	\end{itemize}
\end{lemma}
\begin{proof}
	To prove this, we shall make use of the implicit function theorem. Define a map $F : X_p \times [s_0, s_0 + \delta) \to X_p$ by
	$$
	F(u, s)=u-\frac{1}{-\partial_{xx}  + (-\p_{yy})^s + 1} |u|^{p-2}u.
	$$
	To adapt the implicit function theorem, we first demonstrate that $F$ is well-defined, i.e. for any $u \in X_p$, there holds that $F(u, s) \in X_p$. For this aim, we introduce the fundamental solution $\mathcal{K}$ to the equation
	\begin{align}\label{fequ}
		-\partial_{xx} u + (-\p_{yy})^s  u + u=0.
	\end{align}
	It can be characterized by
	\begin{align} \label{defk}
		\mathcal{K}(x, y)=\int_0^{+\infty} e^{-t} \mathcal{H}(x, y, t) \, dt,
	\end{align}
	where
	\begin{align} \label{defh}
		\mathcal{H}(x,y,t):=\int_{\R} \int_{\R} e^{-2\pi \textnormal{i} (x ,y) \cdot (\xi_1, \xi_2)-t \left(|\xi_1|^2 +|\xi_2|^{2s}\right)} \, d\xi_1d\xi_2:=\mathcal{H}_1(x, t)\mathcal{H}_s(y, t)
	\end{align}
	and
	$$
	\mathcal{H}_1(x, t):=\int_{\R} e^{-2\pi \textnormal{i} x \xi_1 -t |\xi_1|^2} \, d\xi_1, \quad \mathcal{H}_s(y, t):=\int_{\R} e^{-2\pi \textnormal{i}y \xi_2 -t |\xi_2|^{2s}} \, d\xi_2.
	$$
	It is simple to see that
	$$
	\mathcal{H}_1(x, t)=\left(\frac{\pi}{t}\right)^{\frac 12} e^{-\pi\frac{|x|^2}{t}}.
	$$
	Moreover, from $(A 4)$ in \cite[Appendix A]{FQT}, we find that
	$$
	0 <\mathcal{H}_{s}(y ,t) \lesssim \min \left\{t^{-\frac{1}{2s}}, t|y|^{-1-2s}\right\}.
	$$
	As a consequence, we derive that, for any $q>1$,
	\begin{align} \label{nq}
		\begin{split}
			\|\mathcal{H}(\cdot, t)\|_q &=\|\mathcal{H}_1(\cdot,t) \mathcal{H}_s(\cdot,t) \|_q \\
			&\lesssim t^{-\frac 12-\frac{1}{2s}}  \left(\int_{\R} \int_{|y| \leq t^{\frac{1}{2s}}} e^{-q\pi\frac{|x|^2}{t}} \, dxdy\right)^{\frac 1q} + t^{\frac 12}\left(\int_{\R} \int_{|y|>t^{\frac{1}{2s}}} e^{-q\pi\frac{|x|^2}{t}}  |y|^{-q-2sq} \, dxdy\right)^{\frac 1q} \\
			& \lesssim t^{-\left(1-\frac 1q \right)\left(\frac 12+\frac{1}{2s}\right)}.
		\end{split}
	\end{align}
	It then follows that
	$$
	\|\mathcal{K}\|_q \lesssim \int_0^{+\infty} e^{-t} t^{-\left(1-\frac 1q \right)\left(\frac 12+\frac{1}{2s}\right)} \, dt<+\infty, 
	$$
	where $q>1$ satisfies
	\begin{align} \label{qc}
		\left(1-\frac 1q \right)\left(\frac 12+\frac{1}{2s}\right)<1.
	\end{align}
	Applying Young's inequality under the condition \eqref{qc}, we obtain that
	$$
	\left\|\frac{1}{-\partial_{xx}  + (-\p_{yy})^s + 1} |u|^{p-2}u\right\|_2 = \left\|\mathcal{K} \ast |u|^{p-2} u\right\|_2 \lesssim \|u\|_p^{p-1},
	$$
	where
	$$
	1-\frac 1q=\frac 12 -\frac 1 p<\frac{s_0}{1+s_0} \leq \frac{s}{1+s}, \quad 2<p<p_{s_0},
	$$
	and \eqref{qc} is fulfilled. Furthermore, we get that
	$$
	\left\|\frac{1}{-\partial_{xx}  + (-\p_{yy})^s + 1} |u|^{p-2}u\right\|_p = \left\|\mathcal{K} \ast |u|^{p-2} u\right\|_p \lesssim \|u\|_p^{p-1},
	$$
	where
	$$
	1-\frac 1q=1 -\frac 2 p<\frac{2s_0}{1+s_0} \leq \frac{2s}{1+s}, \quad 2<p<p_{s_0},
	$$
	and \eqref{qc} is fulfilled as well. Therefore, we derive that $F$ is well-defined.
	
	Next we are going to prove that $F$ is of class $C^1$. To do this, we first verify that $\frac{\partial F}{\partial u}$ exists and
	$$
	\frac{\partial F}{\partial u}=1-\frac{1}{-\partial_{xx}  + (-\p_{yy})^s + 1} \left(p-1\right)|u|^{p-2}.
	$$
	Define
	$$
	G(u, s):=\frac{1}{-\partial_{xx}  + (-\p_{yy})^s + 1}|u|^{p-2} u.
	$$
	To achieve the desired result, it is equivalent to verify that $\frac{\partial G}{\partial u}$ exists and
	$$
	\frac{\partial G}{\partial u}=\frac{1}{-\partial_{xx}  + (-\p_{yy})^s + 1} \left(p-1\right)|u|^{p-2}.
	$$
	It follows from Sobolev's inequality that
	\begin{align}\label{sbl}
		\|u\|_{X_p} \lesssim \left\|\left(\partial_{x}  + (-\p_{yy})^{\frac{s_p}{2}} + 1\right) u\right\|_2,
	\end{align}
	where
	$$
	s_p:=\frac{p-2}{p+2}<s_0\leq s.
	$$
	Accordingly, invoking Plancherel's identity and Young's inequality, we can infer that, for any $h \in X_p$,
	\begin{align*}
		&\left\|G(u+h, s)-G(u, s)-\frac{\partial G}{\partial u}(u, s) h\right\|_{X_p} \\
		&\lesssim \left\|\left(\partial_{x}  + (-\p_{yy})^{\frac{s_p}{2}} + 1\right)\left(G(u+h, s)-G(u, s)-\frac{\partial G}{\partial u}(u, s) h\right)\right\|_2 \\
		& = \left\|\left(\partial_{x}  + (-\p_{yy})^{\frac{s_p}{2}} + 1\right) \left(\frac{1}{-\partial_{xx}  + (-\p_{yy})^s + 1} \left(|u+h|^{p-2}(u+
		h)-|u|^{p-2} u -(p-1)|u|^{p-2} h\right)\right)\right\|_2 \\
		& \lesssim \left\| \left(\frac{1}{-\partial_{xx}  + (-\p_{yy})^{s-\frac{s_p}{2}} + 1} \left(|u+h|^{p-2}(u+
		h)-|u|^{p-2} u -(p-1)|u|^{p-2} h\right)\right)\right\|_2 \\
		& \lesssim \left\||u+h|^{p-2}(u+h)-|u|^{p-2} u -(p-1)|u|^{p-2} h\right\|_{\frac{p}{p-1}}=o(\|h\|_{X_p}).
	\end{align*}
	Thus the desire result follows immediately.
	
	To achieve that $F$ is of class $C^1$, we further check that $\frac{\partial G}{\partial u}$ is continuous. For this, we need to show that, for any $\eps>0$, there exists $\delta>0$ such that $\|u-\tilde{u}\|_{X_p} +|s-\tilde{s}| <\delta$, then
	$$
	\left\|\left(\frac{\partial G}{\partial u}(u, s)-\frac{\partial G}{\partial u}(\tilde{u}, \tilde{s})\right) h\right\|_{X_p}<\eps \|h\|_{X_p}.
	$$
	In the spirit of \eqref{sbl}, it suffices to assert that
	\begin{align} \label{ctu}
		\left\|\left(A_s |u|^{p-2}- A_{\tilde{s}} |\tilde{u}|^{p-2}\right) h\right\|_2 \leq \eps \|h\|_{X_p},
	\end{align}
	where
	$$
	A_s:=\frac{\partial_{x}  + (-\p_{yy})^{\frac{s_p}{2}} + 1}{-\partial_{xx}  + (-\p_{yy})^s + 1}, \quad A_{\tilde{s}}:=\frac{\partial_{x}  + (-\p_{yy})^{\frac{s_p}{2}} + 1}{-\partial_{xx}  + (-\p_{yy})^{\tilde{s}} + 1}.
	$$
	Observe that
	$$
	\left(A_s |u|^{p-2}- A_{\tilde{s}} |\tilde{u}|^{p-2}\right) = \left(A_s -A_{\tilde{s}}\right) |u|^{p-2} + A_{\tilde{s}}\left(|u|^{p-2}-|\tilde{u}|^{p-2}\right).
	$$
	By Plancherel’s identity, mean value theorem, Young's inequality and H\"older's inequality, we conclude that
	\begin{align*}
		\left\|\left(A_s -A_{\tilde{s}}\right) |u|^{p-2} h\right\|_2 &\lesssim |s-\tilde{s}| \left\|\frac{1}{\left(-\partial_{xx}\right)^{\frac 12}  + (-\p_{yy})^{s+\tilde{s}-\sigma -\frac{s_p}{2}}+1} |u|^{p-2} h\right\|_2 \\
		&\lesssim |s-\tilde{s}| \|u\|_p^{p-2}\|h\|_p,
	\end{align*}
	where $\sigma=\theta s + (1-\theta) \tilde{s} \in [\min\{s, \tilde{s}\}, \max\{s, \tilde{s}\}]$ for $0 \leq \theta \leq 1$ arising from the use of mean value theorem. Furthermore, we can similarly derive that
	\begin{align*}
		\left\|A_{\tilde{s}}\left(|u|^{p-2}-|\tilde{u}|^{p-2}\right) h\right\|_2 &\lesssim \left\|\frac{1}{\partial_{x}  + (-\p_{yy})^{\tilde{s}-\frac{s_p}{2}}+1}\left(|u|^{p-2}-|\tilde{u}|^{p-2}\right) h\right\|_2 \\
		& \lesssim \left\||u|^{p-2}-|\tilde{u}|^{p-2}\right\|_{\frac{p}{p-2}} \|h\|_p.
	\end{align*}
	Note that
	$$
	\left\||u|^{p-2}-|\tilde{u}|^{p-2}\right\|_{\frac{p}{p-2}} \leq \left\||u-\tilde{u}|^{p-2}\right\|_{\frac{p}{p-2}}=\left\|u-\tilde{u}\right\|_{p}^{p-2}, \quad 2<p \leq 3
	$$
	and
	$$
	\left\||u|^{p-2}-|\tilde{u}|^{p-2}\right\|_{\frac{p}{p-2}} \lesssim \left\|\left(|u|^{p-3}+|\tilde{u}|^{p-3}\right)|u-\tilde{u}|\right\|_{\frac{p}{p-2}} \leq \left(\|u\|_p^{p-3} +\|\tilde{u}\|_p^{p-3}\right) \|u -\tilde{u}\|_p, \quad p>3.
	$$
	Therefore, from the discussions above, we obtain that \eqref{ctu} holds true, i.e. $\frac{\partial{G}}{\partial u}$ is continuous. By a similar way, we are able to show that $\frac{\partial{G}}{\partial s}$ exists and
	$$
	\frac{\partial G}{\partial s}=\frac{(-\p_{yy})^{s} \log \left(-\p_{yy}\right) }{-\partial_{xx}  + (-\p_{yy})^{s}+1} |u|^{p-2}u.
	$$
	Further, we can prove that $\frac{\partial{G}}{\partial s}$ is continuous as well. Thus we conclude that $F$ is of class $C^1$.
	
	Based on the previous arguments, we are now ready to make use of the implicit function theorem to prove the desired conclusions. Since $u_0 \in X_p$ solves \eqref{201} with $s=s_0$, then $F(u_0, s_0)=0$. Observe that
	$$
	\frac{\partial F}{\partial u}(u_0, s_0)=1 + K, \quad K:=-\frac{1}{-\partial_{xx}  + (-\p_{yy})^{s_0}   + 1} \left(p-1\right)|u_0|^{p-2}.
	$$
	It is not hard to check that $K$ is compact operator on $L^2(\R^2)$. Moreover, by the assumption, we know that $-1 \not\in \sigma(K)$, because $0$ is not an eigenvalue of the operator $\mathcal{L}_{+, s_0}$. It then follows that $1+K$ is invertible on $L_a^2(\R^2)$. Moreover, arguing as before, with the help of Plancherel’s identity and Young’s inequality, we can derive that $K$ maps $X_p$ to $X_p$ and it is bounded from $X_p$ to $X_p$. This shows that $1+K$ is bounded from $X_p$ to $X_p$. Therefore, we have that $(1+K)^{-1}$ is bounded from $X_p$ to $X_p$. At this point, adapting the implicit function theorem, we can get the desired conclusions and the proof is completed.
\end{proof}

Hereafter, we assume that $0<s_0<1$, $2<p<p_{s_0}$ and $u_0$ be a solution to \eqref{201} with $s=s_0$ satisfying the assumptions of Lemma \ref{bri} and $u_s \in C^1(I, X_p)$ be the solution to \eqref{201} obtained in Lemma \ref{bri}. We now consider the maximum extension of the branch $u_s$ for $s \in [s_0, s_*)$, where $s_*>s_0$ is given by
$$
s_*:=\sup \left\{s_0<\tilde{s}<1, \phi_s \in C^1([s_0, \tilde{s}); X_p), u_s \,\,\mbox{satisfies the assumptions of Lemma \ref{bri} for} \,\, s \in [s_0, \tilde{s}) \right\}.
$$
In the following, we are going to prove that $s_*=1$. To begin with, we show the following a priori bounds for $u_s$.

\begin{lemma} \label{equi}
	There holds that
	$$
	\int_{\R^2} |u_s|^2 \, dx \sim \int_{\R^2}|\partial_ x u_s|^2 \, dxdy +\int_{\R^2} |D_y^s u_s|^2 \, dxdy \sim \int_{\R^2} |u_s|^p \,dxdy \sim 1
	$$
	for any $s \in [s_0, s_*)$.
\end{lemma}
\begin{proof}
	For simplicity, we shall define
	$$
	M_s:=\int_{\R^2} |u_s|^2 \, dx, \quad T_s:=\int_{\R^2}|\partial_ x u_s|^2 \, dxdy +\int_{\R^2} |D_y^s u_s|^2 \, dxdy, \quad V_s:=\int_{\R^2} |u_s|^p \,dxdy.
	$$
	Since $u_s$ is a solution to \eqref{201}, by multiplying \eqref{201} against $u_s$ and integrating over $\R^2$, then
	\begin{align} \label{equi0}
		T_s +M_s=V_s.
	\end{align}
	On the other hand, by Lemma \ref{poh}, there holds that
	\begin{align} \label{equi1}
		sT_s=\frac{(1+s)(p-2)}{2p} V_s.
	\end{align}
	Using \eqref{equi1} and $s \geq s_0$, then
	$$
	\frac{p-2}{2p} V_s<T_s \leq \frac{(1+s_0)(p-2)}{2s_0p} V_s,
	$$
	which leads to $V_s \sim T_s$ for any $[s_0, s_*)$. Furthermore, combining \eqref{equi0} and \eqref{equi1}, we see that
	$$
	\frac{2(s_0+1)-p(1-s_0)}{2ps_0} V_s \leq M_s <\frac{p-2}{2p}V_s.
	$$
	This implies that $V_s \sim M_s$ for any $[s_0, s_*)$. Thus we infer that
	\begin{align} \label{equi2}
		M_s \sim T_s \sim V_s
	\end{align}
	for any $s \in [s_0, s_*)$.
	
	Due to $2<p <p_{s_0}$, then there exists $0<\theta<1$ such that $p=2 \theta + p_{s_0} (1-\theta)$. Therefore, by H\"older's inequality and \eqref{gn}, we obtain that
	\begin{align}  \nonumber
		V_s \leq M_s^{\theta} \left(\int_{\R^2} |u_s|^{p_{s_0}} \, dx \right)^{1-\theta} &\leq C_{s_0} M_s^{\theta} \left(\int_{\R^2}|\partial_ x u_s|^2 \, dxdy +\int_{\R^2} |D_y^{s_0} u_s|^2 \, dxdy\right)^{\frac{p_{s_0}(1-\theta)}{2}} \\ \nonumber
		&\leq C_{s_0} M_s^{\theta} \left(\int_{\R^2}|\partial_ x u_s|^2 \, dxdy +\int_{\R^2}  |D_y^{s} u_s|^2 \, dxdy +\int_{\R^2} |u_s|^2 \, dxdy \right)^{\frac{p_{s_0}(1-\theta)}{2}} \\ \label{equi11}
		&=C_{s_0} M_s^{\theta} \left(T_s+M_s\right)^{\frac{p_{s_0}(1-\theta)}{2}},
	\end{align}
	where we used the fact that
	$$
	\int_{\R^2} |D_y^{s_0} u_s|^2 \, dxdy \lesssim \int_{\R^2}  |D_y^{s} u_s|^2 \, dxdy +\int_{\R^2} |u_s|^2 \, dxdy, \quad 0<s_0 \leq s.
	$$
	It then follows from \eqref{equi11} that $V_s \geq 1$ for any $s \in [s_0, s_*)$, because of $p>2$. Using \eqref{equi2}, we then have that
	$$
	M_s \sim T_s \sim V_s \gtrsim 1,
	$$
	for any $s \in [s_0, s_*)$.
	
	In what follows, we shall demonstrate that
	\begin{align} \label{a1}
		M_s \sim T_s \sim V_s \lesssim 1.
	\end{align}
	For this, we first verify that
	\begin{align} \label{claim1}
		\left\|\left(-\partial_{yy}\right)^{t} u_s\right\|_2^2 \lesssim V^{\frac{2(p-1)}{p}}_s,
	\end{align}
	where
	$$
	t:=s-\frac{p-2}{4p}>0, \quad 2<p<p_{s_0},  \quad s \geq s_0.
	$$
	Since $u_s$ solves \eqref{201}, then
	$$
	u_s=\frac{1}{-\partial_{xx}  + (-\p_{yy})^{s}   + 1}|u_s|^{p-2}u_s.
	$$
	It then follows from Plancherel’s identity that
	$$
	\left\|\left(-\partial_{yy}\right)^{t} u_s\right\|_2^2=\left\|\frac{\left(-\partial_{yy}\right)^{t}}{-\partial_{xx}  + (-\p_{yy})^{s}   + 1}|u_s|^{p-2}u_s\right\|_2^2 \leq \left\| (-\p_{yy})^{t-s}|u_s|^{p-2}u_s\right\|_2^2.
	$$
	Since $0<t<s$, by using the fundamental solution to the equation $(-\p_{yy})^{-(s-t)} u=0$, then
	$$
	\left\| (-\p_{yy})^{t-s}|u_s|^{p-2}u_s\right\|_2^2 = c_p\left\||\cdot|^{\frac{p+2}{2p}}\ast |u_s|^{p-2}u_s\right\|_2^2,
	$$
	where $c_p>0$ is a constant depending only on $p$. As a consequence of Young's inequality for weak type space, we then get that
	$$
	\left\| (-\p_{yy})^{t-s}|u_s|^{p-2}u_s\right\|_2^2 \lesssim \||u_s|^{p-1}\|_{\frac{p}{p-1}}^2=V_s^{\frac{2(p-1)}{p}}.
	$$
	Thus \eqref{claim1} follows. Now we differentiate \eqref{201} satisfied by $u_s$ with respect to $s$ to infer that
	\begin{align} \label{d1}
		\mathcal{L}_{+, s} \dot{u}_s=-(-\partial_{yy})^{s} \log\left(-\partial_{yy}\right) u_s, \quad \dot{u}_s:=\frac{d}{ds}u_s,
	\end{align}
	Since $u_s$ is a solution to \eqref{201}, then
	\begin{align} \label{d2}
		\mathcal{L}_{+, s} u_s=-(p-2)|u_s|^{p-2}u_s.
	\end{align}
	Accordingly, by invoking \eqref{d1} and \eqref{d2}, we get that
	\begin{align*}
		\frac{d}{ds}V_s=p \int_{\R^2} u_s^{p-2} u_s \dot{u}_s \,dxdy&=-\frac{p}{p-2} \int_{\R^2} \mathcal{L}_{+, s} u_s \dot{u}_s \,dxdy \\
		&=-\frac{p}{p-2} \int_{\R^2} u_s \mathcal{L}_{+, s} \dot{u}_s \,dxdy=\int_{\R^2} u_s (-\partial_{yy})^{s} \log\left(-\partial_{yy}\right) u_s \,dxdy.
	\end{align*}
	Observe that
	\begin{align} \label{equi4}
		\begin{split}
			&\int_{\R^2} u_s (-\partial_{yy})^{s} \log\left(-\partial_{yy}\right) u_s \,dxdy=2\int_{\R^2}|\xi_2|^{2s} \log (|\xi_2|) |\hat{u}_s|^2 \, d\xi_1d\xi_2\\
			&=2\int_{\R}\int_{|\xi_2| \leq R}|\xi_2|^{2s} \log (|\xi_2|) |\hat{u}_s|^2 \, d\xi_1d\xi_2 +2\int_{\R}\int_{|\xi_2| >R}|\xi_2|^{2s} \log (|\xi_2|) |\hat{u}_s|^2 \, d\xi_1d\xi_2 \\
			&=2\log R \int_{\R}\int_{|\xi_2| \leq R}|\xi_2|^{2s} |\hat{u}_s|^2 \, d\xi_1d\xi_2 +2\int_{\R}\int_{|\xi_2| >R}|\xi_2|^{2s} \log (|\xi_2|) |\hat{u}_s|^2 \, d\xi_1d\xi_2 \\
			& \leq 2(\log R)T_s + 2R^{2s-4t} \log R \int_{\R^2}|\xi_2|^{4t} |\hat{u}_s|^2 \, d\xi_1d\xi_2 \\
			& \lesssim 2(\log R)T_s + 2R^{2s-4t} (\log R) V_s^{\frac{2(p-1)}{p}},
		\end{split}
	\end{align}
	where we used \eqref{claim1} and the following fact for the inequalities,
	\begin{align} \label{m}
		|\xi_2|^{2s-4t} \log (|\xi_2|) \leq R^{2s-4t} \log R, \quad |\xi_2| \geq R,
	\end{align}
	for $R \geq e^{\frac{1}{4t-2s}}$. Indeed, it is simple to compute that \eqref{m} benefits from the monotonicity of the function $f(t):=t^{2s-4t} \log t$ for $t>0$. Thanks to $V_s \gtrsim 1$ for any $[s_0, s_*)$ and $\frac{2(p-1)}{p}>1$, then there exists a suitable $c \sim 1$ such that
	$$
	R^{4t-2s}=cV_s^{\frac{2(p-1)}{p}-1} \geq e.
	$$
	Note that $T_s \sim V_s$ for any $[s_0, s_*)$, from \eqref{equi4}, then
	$$
	\frac{d}{ds} V_s \lesssim \left(1+ \log V_s\right) V_s.
	$$
	This gives that
	$$
	V_s \lesssim e^{e^s} < e^{e^{s_*}}\lesssim 1
	$$
	for any $s \in [s_0, s_*)$, which infers that \eqref{a1} holds true. Thus the proof is completed.
\end{proof}

\begin{lemma} \label{decay}
	For any $s \in[s_0, s_*)$, there holds that $\phi_s>0$. Moreover, for any $s \in[s_0, s_*)$, there holds that
	\begin{align} \label{estimate}
		\phi_s(x, y) \lesssim \left(1+|y|\right)^{-1-2s} e^{-\theta |x|}, \quad (x ,y) \in \R^2,
	\end{align}
	where $0<\theta<1$ is a constant.
\end{lemma}
\begin{proof}
	The estimate \eqref{estimate} directly follows from \cite[Theorem 1.1]{FW}. To complete the proof,  we only need to show that $\phi_s>0$ for any $s \in [s_0, s_*)$. To do this, we begin with asserting that if $u_{\tilde{s}}>0$ for some $\tilde{s} \in [s_0, s_*)$, then there exists $\eps>0$ such that $\phi_s>0$ for any $s \in[s_0, s_*)$ and $|s-\tilde{s}| <\eps$.
	For this, we first need to demonstrate that $\mathcal{L}_{-,s}$ enjoys Perron-Frobenius type property for any $s \in [s_0, s_*)$, i.e. if $\lambda:=\inf \, \sigma(\mathcal{L}_{-,s})$ is an eigenvalue, then $\lambda$ is simple and the corresponding eigenfunction is positive, where
	$$
	\mathcal{L}_{-,s}:=-\partial_{xx}  + (-\partial_{yy})^s   + 1-|u_s|^{p-2}.
	$$
	In the spirit of \cite[Theorem XIII.43]{RS}, the essential argument to prove this lies in checking that $e^{-t \left(-\partial_{xx}  + (-\partial_{yy})^s\right)}$ acting on $L^2(\R^2)$ is positivity improving, i.e. $e^{-t \left(-\partial_{xx}  + (-\partial_{yy})^s\right)} f>0$ for any $f \in L^2(\R^2)$ with $f \not\equiv 0$ and $f \geq 0$. Note that
	\begin{align} \label{positive}
		e^{-t \left(-\partial_{xx}  + (-\partial_{yy})^s\right)}=e^{-t \left(-\partial_{xx}\right)} e^{-t\left(-\partial_{yy}\right)^s}.
	\end{align}
	From \cite[Lemma C.2] {FL}, we know that $e^{-t\left(-\partial_{yy}\right)^s}$ acting on $L^2(\R)$ is positivity improving. Furthermore, invoking Bernstein's theorem, see \cite[Theorem 1.4]{SSV}, we have that
	$$
	e^{-t(-\partial_{xx})}=\int_0^{+\infty} e^{\tau \partial_{xx}} \, d\mu_t(\tau),
	$$
	where $\mu_t$ is a nonnegative measure depending on $t$. This indicates that $e^{-t(-\partial_{xx})}$ acting on $L^2(\R)$ is positivity improving. Therefore, from \eqref{positive}, we conclude that $e^{-t \left(-\partial_{xx}  + (-\partial_{yy})^s\right)}$ acting on $L^2(\R^2)$ is positivity improving.
	
	Next we need to show that if $s_n \to s$ as $n \to \infty$, then $\mathcal{L}_{-,s_n} \to \mathcal{L}_{-,s}$ as $n \to \infty$ in norm-resolvent sense, i.e.
	\begin{align} \label{nrc}
		\left\|\frac{1}{\mathcal{L}_{-, s_n}+z} -\frac{1}{\mathcal{L}_{-, s}+z}\right\|_{L^2 \to L^2}=o_n(1),
	\end{align}
	where $z \in \mathbb{C}$ and $\mbox{Im} \, z \neq 0$.
	Indeed, according to \cite[Theorem 4.3]{SSL}, to achieve \eqref{nrc}, we only need to show that \eqref{nrc} holds for some $z \in \mathbb{C}$ and $\mbox{Im} \, z \neq 0$ close to the origin. Note first that
	\begin{align*}
		\mathcal{L}_{-,s_n}+z &=-\partial_{xx}  + (-\partial_{yy})^{s_n}  + (1+z)-|u_{s_n}|^{p-2} =\left(1-\mathcal{A}_{s_n}\right)\left(-\partial_{xx}  + (-\partial_{yy})^{s_n} + (1+z)\right), 
	\end{align*}
	and
	\begin{align*}
		\mathcal{L}_s +z &=-\partial_{xx}  + (-\partial_{yy})^s + (1+z)-\phi_s^{p-2} =\left(1-\mathcal{A}_s\right)\left(-\partial_{xx}  + (-\partial_{yy})^s + (1+z)\right), 
	\end{align*}
	where
	$$
	\mathcal{A}_{s_n}:=|u_{s_n}|^{p-2}\frac{1}{-\partial_{xx}  + (-\partial_{yy})^{s_n} + (1+z)}, \quad \mathcal{A}_s:=\phi_s^{p-2}\frac{1}{-\partial_{xx}  + (-\partial_{yy})^s + (1+z)}.
	$$
	Since $u_{s_n}$ is a solution to \eqref{201}, then $\mathcal{L}_{-, s_n} u_{s_n}=0$, which readily suggests that $0$ is an eigenvalue of $\mathcal{L}_{-, s_n}$ and $u_{s_n}$ is the corresponding eigenfunction. Notice that $\mathcal{L}_{-,s_n}$ is nonnegative, then $0$ is the lowest eigenvalue of $\mathcal{L}_{-, s_n}$. Since $\mathcal{L}_{-,s}$ enjoys Perron-Frobenius type property, then $0$ is simple and $u_{s_n}>0$. Applying again the fact that $\mathcal{L}_{-, s_n}$ enjoys Perron-Frobenius type property and arguing by contradiction, one can check that $0$ is isolated eigenvalue of $\mathcal{L}_{-, s_n}$. Therefore, we can choose a proper $z \in \mathbb{C}$ and $\mbox{Im} \, z \neq 0$ close to the origin such that $-z$ is not an eigenvalue of $\mathcal{L}_{-, s_n}$. Then we obtain that $0$ is not an eigenvalue of $1-\mathcal{A}_{s_n}$. In addition, it is not difficult to verify that $\mathcal{A}_{s_n}$ is compact from $L^2(\R^2)$ to $L^2(\R^2)$. As a result, we conclude that $1-\mathcal{A}_{s_n}$ is invertible from $L^2(\R^2)$ to $L^2(\R^2)$ for any $n \in \mathbb{N}$. Similarly, we can show that $1-\mathcal{A}_s$ is invertible from $L^2(\R^2)$ to $L^2(\R^2)$ as well. Now we can write
	$$
	\frac{1}{\mathcal{L}_{-,s_n}+z} =\frac{1}{-\partial_{xx}  + (-\partial_{yy})^{s_n} + (1+z)}\left(1-\mathcal{A}_{s_n}\right)^{-1}
	$$
	and
	$$
	\frac{1}{\mathcal{L}_s+z} =\frac{1}{-\partial_{xx}  + (-\partial_{yy})^s + (1+z)}\left(1-\mathcal{A}_s\right)^{-1}.
	$$
	By means of H\"ormander-Mikhlin's theorem, there holds that
	\begin{align} \label{hm}
		&\left\|\frac{1}{-\partial_{xx}  + (-\partial_{yy})^{s_n} + (1+z)} -\frac{1}{-\partial_{xx}  + (-\partial_{yy})^s + (1+z)}\right\|_{L^2 \to L^2}=o_n(1).
	\end{align}
	Thus, to prove \eqref{nrc}, it suffices to verify that $ \left\|\mathcal{A}_{s_n} - \mathcal{A}_s \right\|_{L^2 \to L^2}=o_n(1)$.
	Note that $u_{s_n}>0$, from \eqref{estimate}, then $\|u_{s_n}\|_{\infty} \lesssim 1$. Taking advantage of \eqref{hm}, we then know that
	\begin{align} \label{c00}
		\begin{split}
			&\left\||u_{s_n}|^{p-2}\left(\frac{1}{-\partial_{xx}  + (-\partial_{yy})^{s_n} + (1+z)} -\frac{1}{-\partial_{xx}  + (-\partial_{yy})^s + (1+z)}\right)\right\|_{L^2 \to L^2} \\
			&\lesssim \|u_{s_n}\|_{\infty}^{p-2}\left\|\left(\frac{1}{-\partial_{xx}  + (-\partial_{yy})^{s_n} + (1+z)} -\frac{1}{-\partial_{xx}  + (-\partial_{yy})^s + (1+z)}\right)\right\|_{L^2 \to L^2}=o_n(1).
		\end{split}
	\end{align}
	Observe that
	\begin{align*}
		\frac{1}{-\partial_{xx}  + (-\partial_{yy})^s + (1+z)}=\int_{0}^{+\infty} e^{-(1+z)t} e^{-t\left(-\partial_{xx}  + (-\partial_{yy})^s\right)} \, dt,
	\end{align*}
	where
	$$
	e^{-t\left(-\partial_{xx}  + (-\partial_{yy})^s\right)} f=\mathcal{H}(x, y, t) \ast f.
	$$
	and $\mathcal{H}$ is defined by \eqref{defh}.
	This results in
	\begin{align} \label{c0}
		\begin{split}
			&\left\|\left(|u_{s_n}|^{p-2} -|u_s|^{p-2}\right) \frac{1}{-\partial_{xx}  + (-\partial_{yy})^s + (1+z)} \right\|_{L^2 \to L^2}\\ &\leq \int_{0}^{+\infty} e^{-(1+z)t} \left\| \left(|u_{s_n}|^{p-2} -|u_s|^{p-2}\right)e^{-t\left(-\partial_{xx}  + (-\partial_{yy})^s\right)}\right\|_{L^2 \to L^2} \, dt \\
			& \leq  \left\||u_{s_n}|^{p-2} -|u_s|^{p-2}\right \|_q  \int_{0}^{+\infty} e^{-(1+z)t} \|\mathcal{H}(\cdot, t)\|_{\frac{q}{q-1}} \, dt \\
			& \lesssim \left\| |u_{s_n}|^{p-2} -|u_s|^{p-2}\right \|_q \int_{0}^{+\infty} e^{-(1+z)t} t^{-\frac 1 q \left(\frac 12 +\frac{1}{2s}\right) } \, dt  \lesssim \left\| |u_{s_n}|^{p-2} -|u_s|^{p-2}\right \|_q,
		\end{split}
	\end{align}
	where we used H\"older's inequality, Young's inequality and \eqref{nq} for $q>2$ such that
	\begin{align*}
		\frac 1 q \left(\frac 12 +\frac{1}{2s}\right)<1.
	\end{align*}
	Note that
	\begin{align} \label{c1}
		\left\||u_s|^{p-2}-|u_{s_n}|^{p-2}\right\|_{q} \leq \left\||u_s-u_{s_n}|^{p-2}\right\|_{q}=\left\|u-u_{s_n}\right\|_{q(p-2)}^{p-2}, \quad 2<p \leq 3
	\end{align}
	and
	\begin{align} \label{c2}
		\begin{split}
			\left\||u_s|^{p-2}-|u_{s_n}|^{p-2}\right\|_{q} & \lesssim \left\|\left(|u_s|^{p-3}+|u_{s_n}|^{p-3}\right)|u-u_{s_n}|\right\|_{q} \\
			&\leq \left(\|u_s\|_{\infty}^{p-3} +\|u_{s_n}\|_{\infty}^{p-3}\right) \|u_s-u_{s_n}\|_q \lesssim \|u_s-u_{s_n}\|_q, \quad p>3,
		\end{split}
	\end{align}
	where we used the fact that $\|u_{s_n}\|_{\infty} \lesssim 1$ and $u_s \in L^{\infty}(\R^2)$. Since $s_n \geq s_0>0$ and $\{u_{s_n}\} \subset H^{1, s_n}(\R^2)$ by Lemma \ref{regularity}, it then follows from interpolation arguments that $\{u_{s_n}\} \subset H^{1, s_0}(\R^2)$. Further, in light of Lemma \ref{equi}, then $\{u_{s_n}\} \subset H^{1,s_0}(\R^2)$ is bounded. Therefore, we see that there exists $u_{s_n} \rightharpoonup u_s$ in $H^{1,s_0}(\R^2)$ as $n \to \infty$. Since  $H^{1, s_0}(\R^2)$ is compactly embedded into $L_{loc}^2(\R^2)$, then $u_{s_n} \to u_s$ in $L^2_{loc}(\R^2)$ as $n \to \infty$. From \eqref{estimate}, then $u_{s_n} \to u_s$ in $L^2(\R^2)$ as $n \to \infty$.
	Since $\|u_{s_n}\|_{\infty} \lesssim 1$, then $u_{s_n} \to u_s$ in $L^q(\R^2)$ for any $q>2$ as $n \to \infty$. As a consequence, from \eqref{c1} and \eqref{c2}, it yields that
	$$
	\left\||u|^{p-2}-|u_{s_n}|^{p-2}\right\|_{q}=o_n(1)
	$$
	for any $q>2$. Going back to \eqref{c0}, we then have that
	$$
	\left\|\left(|u_{s_n}|^{p-2} -|u_s|^{p-2}\right) \frac{1}{-\partial_{xx}  + (-\partial_{yy})^s + (1+z)} \right\|_{L^2 \to L^2}=o_n(1).
	$$
	This along with \eqref{c00} leads to $\left\|\mathcal{A}_{s_n} - \mathcal{A}_s \right\|_{L^2 \to L^2}=o_n(1)$. Thus the desired conclusion follows.
	
	Based on discussions above, we are now able to prove the assertion. Since $u_{\tilde{s}}$ is a solution to \eqref{201} with $s=\tilde{s}$, then $\mathcal{L}_{-,\tilde{s}}u_{\tilde{s}}=0$. This suggests that $u_{\tilde{s}}$ is an eigenfunction of $\mathcal{L}_{-,\tilde{s}}$ with eigenvalue $0$. Due to $u_{\tilde{s}}>0$, then $0$ is the lowest eigenvalue of $\mathcal{L}_{-,\tilde{s}}$, i.e. $\lambda_1(\mathcal{L}_{-,\tilde{s}})=0$. Further, we can infer that $0$ is an isolated eigenvalue.
	It follows from norm-resolvent convergence of the operator $\mathcal{L}_{-, s}$ that $\lambda_1(\mathcal{L}_{-,s}) \to \lambda_1(\mathcal{L}_{-, \tilde{s}})$ as $s \to \tilde{s}$. Thanks to $\lambda_1(\mathcal{L}_{-,\tilde{s}})=0$ and it is isolated, then there exists $\eps>0$ such that $\lambda_1(\mathcal{L}_{-,s})=0$ and it is isolated for $|s-\tilde{s}|<\eps$. Note that $\phi_s \to u_{\tilde{s}}$ in $L^2(\R^2)$ as $s \to \tilde{s}$ by norm-resolvent convergence of the operator $\mathcal{L}_{-, s}$. Applying again Perron-Frobenius type property of the operator $\mathcal{L}_{-,s}$ and the fact that $u_{\tilde{s}}>0$, we then derive that $\phi_s>0$ for any $|s-\tilde{s}|<\eps$.
	
	Next we verify that if $\phi_s>0$ for any $s \in [s_0, \tilde{s})$ for some $\tilde{s}<s_*$, then $u_{\tilde{s}}>0$. Let $\tilde{s} \in (s_0, s_*)$ and $\{s_n\} \subset [s_0, \tilde{s})$ be a sequence such that $s_n \to \tilde{s}$ as $n \to \infty$. Since $u_{s_n}>0$ and $u_{s_n} \to u_{\tilde{s}}$ in $L^2(\R^2)$ as $n \to \infty$, then $u_{\tilde{s}} \geq 0$. Note that $u_{\tilde{s}} \not\equiv 0$ by Lemma \ref{equi} and $u_{\tilde{s}} \geq 0$, then $u_{\tilde{s}}>0$ by maximum principle.
	To summarize, there holds that $\phi_s>0$ for any $s \in [s_0, s_*)$. This completes the proof.
\end{proof}

\begin{lemma} \label{conv}
	Let $\{s_n\} \subset [s_0, s_*)$ be a sequence such that $s_n \to s_*$ as $n \to \infty$ and $u_{s_n}>0$ for any $n \in \mathbb{N}$. Then there exists $u_* \in L^2(\R^2) \cap L^p(\R^2)$ such that $u_{s_n} \to u_*$ in $L^2(\R^2) \cap L^p(\R^2)$ as $n \to \infty$. Moreover, there holds that $u_*>0$ and it solves the equation
	\begin{align} \label{2011}
		-\partial_{xx} u_* + (-\partial_{yy})^s u_*  + u_*=u_*^{p-1}.
	\end{align}
\end{lemma}
\begin{proof}
	For simplicity, we shall write $u_n=u_{s_n}$. Note that $\{u_n\} \subset H^{1, s_n}(\R^2)$ by Lemma \ref{regularity}, then $\{u_n\} \subset H^{1, s_0}(\R^2)$, because of $s_n \geq s_0>0$. It follows from Lemma \ref{equi} that $\{u_n\} \subset H^{1,s_0}(\R^2)$ is bounded. Then there exists $u_* \in H^{1,s_0}(\R^2)$ such that $u_n \rightharpoonup u_*$ in $H^{1,s_0}(\R^2)$ as $n \to \infty$. Furthermore, by the fact that $H^{1, s_0}(\R^2)$ is compactly embedded into $L_{loc}^2(\R^2)$, we have that $u_n \to u_*$ in $L^2_{loc}(\R^2)$ as $n \to \infty$. Using \eqref{estimate} in Lemma \ref{decay}, we then get that $u_n \to u_*$ in $L^2(\R^2)$ as $n \to \infty$. In view of H\"older's inequality, we then obtain that $u_n \to u_*$ in $L^p(\R^2)$ as $n \to \infty$. Since $u_n \in H^{1, s_n}(\R^2)$ is a solution to \eqref{201} with $s=s_n$, then
	\begin{align} \label{20111}
		u_n=\frac{1}{-\partial_{xx}  + (-\partial_{yy})^{s_n}  + 1}{|u_n|^{p-2}u_n}.
	\end{align}
	Taking into account H\"ormander-Mikhlin's theorem, we obtain that
	\begin{align} \label{s1}
		\left\|\frac{1}{-\partial_{xx}  + (-\partial_{yy})^{s_n} +1} -\frac{1}{-\partial_{xx}  + (-\partial_{yy})^{s_*} +1}\right\|_{L^\frac{p}{p-1} \to L^\frac{p}{p-1}}=o_n(1).
	\end{align}
	In addition, we see that
	\begin{align*}
		\left\||u_n|^{p-2}u_n-|u_*|^{p-2}u_*\right\|_{\frac{p}{p-1}} &\lesssim \left\|\left(|u_n|^{p-2}+|u_*|^{p-2}\right)|u_n-u_*|\right\|_{\frac{p}{p-1}} \\
		&\leq \left(\|u_n\|_p^{p-2} +\|u_*\|_p^{p-2}\right) \|u_n -u_*\|_p=o_n(1),
	\end{align*}
	which suggests that
	\begin{align} \nonumber
		\left\|\frac{1}{-\partial_{xx}  + (-\partial_{yy})^{s_*} +1}\left(|u_n|^{p-2}u_n-|u_*|^{p-2}u_*\right)\right\|_{\frac{p}{p-1}} &= \left\|\mathcal{K} \ast\left(|u_n|^{p-2}u_n-|u_*|^{p-2}u_*\right)\right\|_{\frac{p}{p-1}} \\ \label{s2}
		& \leq \left\||u_n|^{p-2}u_n-|u_*|^{p-2}u_*\right\|_{\frac{p}{p-1}}=o_n(1),
	\end{align}
	where we used Young's inequality and the fact that $\mathcal{K} \in L^1(\R^2)$ is the fundamental solution to the eqution
	$$
	-\partial_{xx} u + (-\partial_{yy})^{s_*} u +u=0.
	$$
	Therefore, from \eqref{s1} and \eqref{s2}, we are able to conclude that
	\begin{align*}
		&\left\|\frac{1}{-\partial_{xx}  + (-\partial_{yy})^{s_n} +1} |u_n|^{p-2}u_n-\frac{1}{-\partial_{xx}  + (-\partial_{yy})^{s_*} +1} |u_*|^{p-2}u_*\right\|_{\frac{p}{p-1}} \\
		&\leq \left\| \left(\frac{1}{-\partial_{xx}  + (-\partial_{yy})^{s_n} +1}-\frac{1}{-\partial_{xx}  + (-\partial_{yy})^{s_*} +1} \right) |u_n|^{p-2}u_n\right\|_{\frac{p}{p-1}}  \\
		& \quad + \left\|\frac{1}{-\partial_{xx}  + (-\partial_{yy})^{s_*} +1} \left( |u_n|^{p-2}u_n-|u_*|^{p-2}u_*\right)\right\|_{\frac{p}{p-1}}=o_n(1).
	\end{align*}
	It then yields from \eqref{20111} that $u_*$ solves the equation
	$$
	u_*=\frac{1}{-\partial_{xx}  + (-\partial_{yy})^{s_*}  + 1}{|u_*|^{p-2}u_*},
	$$
	which indicates that $u_*$ solves \eqref{2011}. Note that $u_* \geq 0$ and $u^* \not \equiv 0$, because of $u_{s_n} >0$ and Lemma \ref{equi}. By maximum principle, then $u_*>0$. Thus the proof is completed.
\end{proof}

\begin{lemma} \label{extend}
	Let $u_0 \in X_p$ be a ground state to \eqref{201}. Then its maximum branch $\phi_s$ with $s \in [s_0, s_*)$ extends to $s_*=1$.
\end{lemma}
\begin{proof}
	Note that $u_0 \in X_p$ is a ground state to \eqref{201}, then the Morse index of $L_{+, s_0}$ acting on $L^2_a(\R^2)$ is $1$, i.e. $\mathcal{N}_{-, a}(\mathcal{L}_{-, s_0})=1$. Reasoning as in the proof of Lemma \ref{decay}, we can also show norm-resolvent convergence of the operator $\mathcal{L}_{+, s}$. This immediately gives that there exists $\eps>0$ such that, for any $s \in [s_0, s_0 +\eps)$,
	$$
	\mathcal{N}_{-, a}(\mathcal{L}_{-, s})=\mathcal{N}_{-, a}(\mathcal{L}_{-, s_0})=1.
	$$
	For such an $\eps>0$, let $s_n \subset [s_0, s_0 +\eps)$ be a sequence such that $s_n \to s_0 + \eps$ as $n \to \infty$. Then $\mathcal{L}_{+, s_n} \to \mathcal{L}_{+, s_0+\eps}$ in the norm resolvent sense as $n \to \infty$. Using the lower semicontinuity of the Morse index with respect to the convergence, we then have that
	$$
	1=\liminf_{n \to \infty}  \mathcal{N}_{-, a}(\mathcal{L}_{+, s_n}) \geq \mathcal{N}_{-, a}(\mathcal{L}_{+, s_0 +\eps}).
	$$
	On the other hand, since $u_{s_0+\eps}$ is a solution to \eqref{201} with $s=s_0+\eps$, then
	$$
	\left(u_{s_0+\eps}, \mathcal{L}_{+, s_0+\eps}u_{s_0+\eps}\right)=-(p-2)\int_{\R^2} |u_{s_0+\eps}|^p \,dxdy<0.
	$$
	Thus we conclude that $\mathcal{N}_{-, a}(\mathcal{L}_{+, s_0 +\eps})=1$. Repeating the discussions above, we are able to infer that $\mathcal{N}_{-, a}(\mathcal{L}_{+, s})=1$ for any $s \in [s_0, s_*]$. In particular, there holds that $\mathcal{N}_{-, a}(\mathcal{L}_{+, s_*})=1$.
	It then follows that $u_* \in X_p$ is a ground state to \eqref{201}. Hence we get that $\mathcal{L}_{+, s_*}$ has trivial kernel on $L_a^2(\R^2)$. If $s_*<1$, then $u_s$ could be extended beyond $s_*$. This contradicts the definition of $s_*$. Thus we conclude that $s_*=1$ and the proof is completed.
\end{proof}

We are now in a positive to present the proof of Theorem \ref{uniqueness}.

\begin{proof}[Proof of Theorem \ref{uniqueness}]
	Let $0<s_0<1$ and $2<p<p_{s_0}$. Suppose that $u_{s_0} \in X_p$ and $\phi_{s_0} \in X_p$ are two different ground states to \eqref{201} with $s=s_0$. In view of Lemmas \ref{bri} and \ref{extend}, then there exist $u_s \in C^1([s_0, 1); X_p)$ and $\phi_s \in C^1([s_0, 1); X_p)$. Moreover, it yields from the local uniqueness in Lemma \ref{bri} that $u_s \neq \phi_s$ for any $s \in [s_0, 1)$. Later, in virtue of Lemma \ref{conv}, we know that there exist $u \in X_p$ and $\phi \in X_p$ such that $u_s \to u$ and $\phi_s \to \phi$ in $L^2(\R^2) \cap L^p(\R^2)$ as $s \to 1^-$. In addition, we have that $u \in X_p$ and $\phi \in X_p$ solve \eqref{201} with $s=1$. Further, arguing as the proof of Lemma \ref{regularity}, we can show that $u \in H^1(\R^2)$ and $\phi \in H^1(\R^2)$. Note that $u_{s}>0$ and $ u_{\tilde{s}} > 0$ for any $s, \tilde{s} \in[s_0, 1)$ by Lemma \ref{decay}. This together with Lemma \ref{equi} leads to $u \not \equiv 0, u \geq 0$ and $\phi \not \equiv 0, \phi \geq 0$. Then, by maximum principle, we can derive that $u>0$ and $\phi>0$. Furthermore, by the moving plane methods, we have that $u, \phi$ are radially symmetric. It then follows from \cite{K} that $u=\phi$. Hence there holds that
	$$
	\|u_s-\phi_s\|_{X_p} \to 0 \quad \mbox{as} \,\, s \to 1.
	$$
	Since the linearized operator
	$$
	\mathcal{L}_+:=-\partial_{xx}  -\partial_{yy} +1-(p-1)u^{p-2}
	$$
	is nondegenerate on $L_{rad}^2(\R^2)$, by using the implicit function theorem under the same spirit as the proof of Lemma \ref{bri}, then there exists a unique branch $\tilde{u}_s \in C^1((1-\eps, 1]; X_p)$ satisfying $\tilde{u}_1=u=\phi$ and solving \eqref{201}. This contradicts the assumption that $u_s \neq \phi_s$ for any $s \in [s_0, 1)$. Thus the proof is completed.
\end{proof}

Finally, we give the proof of Theorem \ref{nd}.
\begin{proof}
	Let $u_1 \in H^1(\R^2)$ be the unique ground state to \eqref{201} with $s=1$. It is well-known that $u_1$ is nondegeneracy, i.e
	$$
	Ker \mathcal{L}_{+, 1} =span \left\{\partial_{x} u_1, \partial_y u_1\right\}.
	$$
	Since $u_1$ is radially symmetric, then $\mathcal{L}_{+, 1}$ has trivial kernel on $L^2_a(\R^2)$. In the spirit of the proof of Lemma \ref{bri}, using the implicit function theorem, we can derive the uniqueness of ground states to \eqref{201} for $s$ close to 1 in $X_p$. Let $u_s$ be the unique ground state to \eqref{201} for $s$ close to 1. Therefore, we see that
	$$
	span \left\{\partial_{x} u_s, \partial_y u_s\right\} \subset  Ker \mathcal{L}_{+, s}.
	$$
	Note that $dim\,Ker \mathcal{L}_{+, 1}=2$ and $\mathcal{L}_{+, s} \to \mathcal{L}_{+, 1}$ in norm-resolvent sense as $s \to 1^-$ by arguing as the proof of Lemma \ref{decay}, then $dim \,Ker \mathcal{L}_{+, s} \leq 2$. Consequence, there holds that $dim \,Ker \mathcal{L}_{+, s} = 2$ and
	$$
	Ker \mathcal{L}_{+, s} =span \left\{\partial_{x} u_s, \partial_y u_s\right\}.
	$$
	Thus the proof is completed.
\end{proof}

\end{document}